\documentclass[11pt,twoside]{article}
\usepackage{graphicx,hyperref,color}
\usepackage[utf8]{inputenc}
\usepackage[margin=1in]{geometry}
\usepackage{marvosym}
\usepackage{cite}
\usepackage{verbatim}
\usepackage{amsmath,amssymb,amsthm}
\usepackage{chngcntr}
\usepackage[stable]{footmisc}
\usepackage{float}
\usepackage{url}
\usepackage{xspace}
\usepackage{paralist}
\usepackage{multienum}
\usepackage{enumitem}
\usepackage{thmtools}
\usepackage{thm-restate}
\usepackage{hyperref}
\usepackage{cleveref}
\usepackage{authblk}
\usepackage{fancyhdr}
\definecolor{orange}{rgb}{1,0.5,0}

\newcommand{\lp}{\mathcal{L}\mathcal{P}}

\newtheorem{theorem}{Theorem}[section]
\newtheorem{corollary}[theorem]{Corollary}
\newtheorem{lemma}[theorem]{Lemma}
\newtheorem{claim}[theorem]{Claim}

\newtheorem{observation}[theorem]{Observation}

\title{A better bound on the largest induced forests in
triangle-free planar graphs}
\author[]{ Hung Le}
\affil[]{Oregon State University}
\date{}
\begin{document}
\maketitle

\begin{abstract} 
It is well-known that there exists a triangle-free planar graph of $n$ vertices such that the largest induced forest has order at most $\frac{5n}{8}$. Salavatipour~\cite{Salavatipour06} proved that there is a forest of order at least $\frac{5n}{9.41}$ in any triangle-free planar graph of $n$ vertices.  Dross, Montassier and Pinlou~\cite{DMP14} improved Salavatipour's bound to $\frac{5n}{9.17}$. In this work, we further improve the bound to  $\frac{5n}{9}$. Our technique is inspired by the recent ideas from Lukot'ka, Maz{\'a}k and Zhu~\cite{LMZ15}.
\end{abstract}

\section{Introduction}
 Albertson and Berman~\cite{AB79} conjectured that every planar graph of $n$ vertices has an induced forest of order at least $\frac{n}{2}$. This conjecture has drawn much attention from graph theory community since it implies that there is an independent set of at least $\frac{n}{4}$ vertices in a planar graph of order $n$; the fact is only known through the Four Color Theorem. However, little progress has been made toward proving this conjecture. Borodin acyclic coloring theorem~\cite{Borodin79} for planar graphs implies the existence of a forest of order at least $\frac{2n}{5}$. To the best of our knowledge, Borodin's result is the best bound for Albertson and Berman conjecture. In the same vein, Akiyama and Watanabe~\cite{AW87} conjectured that a bipartite planar graph of $n$ vertices has an induced forest of order at least $\frac{5n}{8}$. They also presented a bipartite planar graph that has the largest induced forest of order exactly $\lceil \frac{5n}{8} \rceil$. The best bound for the Akiyama and Watanabe conjecture is $\frac{4n}{7}$ due to the recent work by Wang, Xie and Yu~\cite{WXY16}. %We note that the proof in ... is quite complicated and has not been officially published.

Salavatipour~\cite{Salavatipour06} asked the similar question for triangle-free planar graphs. He showed that a triangle-free planar graph of order $n$ has an induced forest of order at least $\frac{17n+24}{32}$, which is approximately $\frac{5n}{9.41}$ (we ignore the additive constant factor as it is insignificant when $n$ is big). Dross, Montassier and Pinlou~\cite{DMP14} improved this bound to $\frac{6n+ 7}{11}$ which is approximately $\frac{5n}{9.17}$. In this work, we further improve this bound to $\frac{5n}{9}$ (Theorem~\ref{thm:main}). We note that Kowalik, Lu\v{z}ar and \v{S}krekovski~\cite{KLS10} obtained  $\frac{5n}{9.01} $ bound which is very closed to our bound, but there is a serious flaw in their proof, as pointed out by Dross, Montassier and Pinlou~\cite{DMP14}. We also note that the example by Akiyama and Watanabe~\cite{AW87} for bipartite planar graphs implies that there exists a triangle-free planar graphs of order $n$ that has the largest induced forest of order at most $\lceil \frac{5n}{8} \rceil$. We believe this bound is a right bound, as evidenced by the work of Alon, Mubayi and Thomas~\cite{AMT01}, who showed that if a triangle-free graph planar graph is cubic, its largest induced forest has order at least $\frac{5n}{8}$. 

%\textbf{To-Do: Add something about other directions} 

\subsection{Previous techniques} \label{sec:prev-tech}
 Here in, we assume that our graph in question, denoted by $G$, is triangle-free. Let $n(G)$ and $m(G)$ be the number of vertices and edges of $G$, respectively. Let $\varphi(G)$ be the order of the largest induced forest in $G$. Previous techniques use discharging to prove:
\begin{equation}\label{eq:prev-op}
 \varphi(G)\geq  a n(G) - b m(G) \quad \mbox{for some appropriate constants $a$ and $b$}
\end{equation}
Since $m(G) \leq 2n(G)-4$ when $G$ is triangle-free planar and $n(G) \geq 3$, Inequality~\ref{eq:prev-op} implies the existence of an induced forest of order at least $(a - 2b)n(G) + 4b$. Salavatipour~\cite{Salavatipour06} proved that Inequality~\ref{eq:prev-op} holds when $(a,b)$ is $(\frac{29}{32}, \frac{6}{32})$, thereby, obtained the bound $\frac{17n(G)+24}{32}$. Dross, Montassier and Pinlou~\cite{DMP14} proved that Inequality~\ref{eq:prev-op} holds when $(a,b)$ is $(\frac{38}{44}, \frac{7}{44})$ and obtained the bound $\frac{6n(G)+ 7}{11}$. Kowalik, Lu\v{z}ar and \v{S}krekovski~\cite{KLS10} tried to modify the Inequality~\ref{eq:prev-op} by adding an additive constant to the right-hand side, but that makes their proof erroneous as noted by Dross, Montassier and Pinlou~\cite{DMP14}. 

To get a good bound on the order of the largest induced forest, one should choose $a$ and $b$ that maximize $(a - 2b)$. However, $a$ and $b$ are constrained by how many vertices one can add to the final induced forest after deleting a subset of vertices and edges of the graph. Roughly speaking, if we delete a set of $\alpha$ vertices, $\beta$ edges from $G$ to obtain a subgraph $G'$ and we can add $\gamma$ vertices from $\alpha$ deleted vertices to the largest induced forest of $G'$ to get an induced forest in $G$, we should choose $a$ and $b$ such that:
\begin{equation}\label{eq:abs-constraint}
a\alpha - b\beta \leq \gamma
\end{equation}
If so, we can apply the inductive proof to show that Inequality~\ref{eq:prev-op} is satisfied as follows:
 
 \begin{equation}
\begin{split}
\varphi(G)& \geq \varphi(G') + \gamma \geq a(n(G)-\alpha) - b(m(G)-\beta) + \gamma\\& \geq a n(G) + b m(G)\\
\end{split}
\end{equation}

This process is repeated until we get down to base cases. As a result, we get a linear program and we need to solve it for $a$ and $b$ that maximize $a - 2b$. For example, Linear Program~\ref{eq:DMP-linear} is from the work of Dross, Montassier and Pinlou~\cite{DMP14}.

\begin{subequations}\label{eq:DMP-linear}
\begin{align}
        b &\geq 0 \\
        0 \leq a &\leq 1 \\
        8a - 12b &\leq 5 \\
        a - 6b &\leq 0 \\
        3a - 10b &\leq 1 
\end{align}
\end{subequations}
We will not try to go into details of Linear Program~\ref{eq:DMP-linear}, but we would like to make a few points that motivate our technique. To get a better bound, one could manage to relax one or more constraints in the linear program. For technical reasons, the first two constraints and the last constraint seems unavoidable. The fourth constraint allows us to only consider graphs of maximum degree at most 5. Thus, one can relax the fourth constraint by considering graphs of higher maximum degree, say 6. But this makes the number of configurations unmanageable. The third constraint, called the \emph{planar cube constraint}, is due to the planar cube (see Figure~\ref{fig:planar-cube}(a)). Specifically, by deleting a planar cube component from $G$, we remove $8$ vertices, $12$ edges and we can only add $5$ vertices back to the forest since the largest induced forest of the planar cube contains $5$ vertices. It turns out that we can relax the planar cube constraint in a different way by introducing two other terms to the right-hand side of Inequality~\ref{eq:prev-op}. Our idea is inspired from the ideas of Lukot'ka, Maz{\'a}k and Zhu~\cite{LMZ15}.

\subsection{Our technique} \label{sec:our-tech}

We use $V(G)$ and $E(G)$ to denote the set of vertices and set of edges, respectively, of $G$. Let $H$ be an induced subgraph of $G$. The degree of $H$, denoted by $\deg_G(H)$, is the number of edges of $G$ with exactly one endpoint in $V(H)$. We use $H^d, H^{d+}$ and $H^{d-}$ to denote an induced subgraph $H$ of degree exactly $d$, at least $d$ and at most $d$, respectively, of graph $G$. Two special graphs of interest in this paper are the planar cube, denoted by $Q_3$, and $K_{3,3}$ minus an edge, denoted by $T_6$ (see Figure~\ref{fig:planar-cube}(b)). The planar cube is a 3-regular planar graph that has $8$ vertices and $12$ edges (see Figure~\ref{fig:planar-cube}(a)). 

\begin{figure}[tbh]
  \centering
   \includegraphics[height=1.2in]{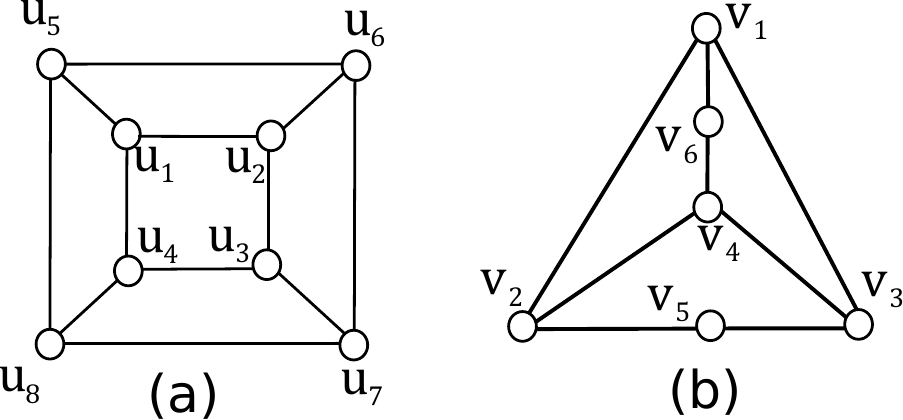}
      \caption{ Two special graphs (a) $Q_3$ and (b) $T_6$.}
  \label{fig:planar-cube}
\end{figure}

%An induced subgraph $H^{d}$ of $G$ is called $Q_3$-like if $H$ is isomorphic to $Q_3$ and $d$ is in $\{0,1\}$.
Let $p(G)$ and $q(G)$ be the maximum number of $Q_3^{1-}$ vertex-disjoint subgraphs and $T_6$ components of $G$, respectively. We will use discharging technique to prove:

	\begin{equation} \label{eq:gen-sub-main}
	\varphi(G) \geq a n(G) - b m(G) -c p(G) - d q(G)
	\end{equation}
for appropriate constants $a,b,c,d$. Essentially, we add two terms depending on $p(G)$ and $q(G)$ to the right-hand side of Inequality~\ref{eq:prev-op}. That would give us more room to find $a$ and $b$ that maximize $a - 2b$. Since $m(G)\leq 2n(G)$ for every triangle-free planar graphs, Inequality~\ref{eq:gen-sub-main} gives us:

\begin{equation}\label{eq:our-opt}
\varphi(G) \geq (a -2b)n(G)  - c p(G) -d q(G)
\end{equation} 
However, we need a bound that is independent of $p(G),q(G)$. This forces us to introduce another technical layer. In the ideal case, both $p(G)$ and $q(G)$ are 0, Inequality~\ref{eq:our-opt} gives us a good bound on $\varphi(G)$. When $p(G) + q(G)$ is at least 1, Lemma~\ref{lm:cublike-forest} and Lemma~\ref{lm:T6-like-forest} allow us to reduce to the ideal case by adding a large portion of vertices from $Q^{1-}_3$ subgraphs and $T_6$ components to the large induced forest.

\subsection{Our results}
Our main result is Theorem~\ref{thm:main} that gives an improved bound on the order of the largest induced forest in triangle-free planar graphs. 

\begin{theorem} \label{thm:main} Every triangle-free planar graph of $n$ vertices contains an induced forest of order at least $\frac{5n}{9}$.
\end{theorem}

We present the full proof of Theorem~\ref{thm:main} in Section~\ref{sec:main-thm-proof}. The main tool in our proof is Theorem~\ref{thm:sub-main} whose proof is deferred to Section~\ref{sec:proof-sub-main}.

\begin{theorem} \label{thm:sub-main} If $a,b,c,d$ are constants that satisfy all constraints in the Linear Program~\ref{eq:our-linear}, then every triangle-free planar graph $G$ has an induced forest of order at least $ an(G) -b m(G) -c p(G) - d q(G)$.
\begin{subequations}\label{eq:our-linear}
\begin{align}
          a &\geq 0 \label{seq:our-iso}\\
         1 - a &\geq 0 \label{seq:our-iso1}\\
         b &\geq 0 \label{seq:our-triv}\\
         c &\geq 0  \label{seq:our-triv1}\\
         d &\geq 0  \label{seq:our-triv2}\\
         1- a + b -c &\geq 0 \label{seq:our-triv3}\\
         1- a + b -d&\geq 0 \label{seq:our-triv4}\\
         5b - a &\geq 0 \label{seq:our-deg-5}\\
         5 - 8a + 12 b +c &\geq 0 \label{seq:Q-3-ex}\\
         4 - 6a + 8 b + d &\geq 0 \label{seq:T-3-ex} \\
         5-8a + 13 b &\geq 0 \label{seq:Q-3-1-ex}\\
         5 + 13 b - 8 a + c - d &\geq 0 \label{seq:Q-3-1-ex1}\\
         4 + 9 b - 6 a  - d &\geq 0 \label{seq:T-6-1-ex}\\
         5-8a + 14b - d  &\geq 0\label{seq:Q-3-2-ex}\\
         5-8a + 14b - c  &\geq 0\label{seq:Q-3-2-ex1}\\
         5-8a + 15b - c - d  &\geq 0\label{seq:Q-3-3-ex}\\
         3-5a + 10b - c  &\geq 0\label{seq:T-3-45-ex}\\
         3-5a + 10b - d  &\geq 0\label{seq:T-3-45-ex1}\\
         3-4a + 4b &\geq 0\label{seq:ex-2-2}
\end{align}
\end{subequations}

\end{theorem}

\begin{corollary} \label{co:sub-main} If $a,b,c,d$ are constants that satisfy all constraints in the Linear Program~\ref{eq:our-linear}, then every triangle-free planar graph $G$ that contains no $Q^{1-}_3$ subgraph and $T_6$ component has an induced forest of order at least $ (a - 2b)n(G) $.
\end{corollary}
\begin{proof}
Since $p(G)$ and $q(G)$ are both $0$, Theorem~\ref{thm:sub-main} implies that $G$ has an induced forest of order at least $a n(G) -b m(G)$. Thus, the corollary follows from the fact that $m(G)\leq 2n(G)$.
\end{proof}
\subsection{Preliminaries}
We define the \emph{order} of $G$ to be $|V(G)|$. Let $\delta(G)$ and $\Delta(G)$ be the minimum and maximum vertex degree of $G$, respectively. We denote the length of a face $f$ by $\ell(f)$. We use $\ell$-face, $\ell^+$-face and $\ell^-$-face to refer to a face of length $\ell$, a face of length at least $\ell$ and a face of length at most $\ell$, respectively. This notation is extended naturally to $\ell$-cycles, $\ell^{-}$-cycles and $\ell^+$-cycles. Similarly, we use $d$-vertex, $d^+$-vertex and $d^-$-vertex to refer to a vertex of degree $d$, a vertex of degree at least $d$ and a vertex of degree at most $d$, respectively.  We reserve $u_1,u_2,u_3,u_4,u_5,u_6,u_7,u_8$ for vertices of $Q_3$ and $v_1,v_2,v_3,v_4,v_5,v_6$ for vertices of $T_6$, as in Figure~\ref{fig:planar-cube}.

Let $H$ be a subgraph of $G$. The \emph{induced embedding} of $H$ from a planar embedding of $G$ is the planar embedding obtained by removing images of vertices and edges not in $H$ from the embedding of $G$. A \emph{between} vertex of $H$ is a vertex that has at least one neighbor outside $H$. We use $G\setminus H$ to denote the subgraph obtained from $G$ by deleting $V(H)$. Let $X$ be a subset of vertices of $H$. We say we can \emph{collect} $X$ if we can add $X$ to any induced forest of $G\setminus H$ to get an induced forest in $G$. A \emph{cut}, denoted by $(V(H), V(G)\setminus V(H))$, is the set of edges with exactly one endpoint in $V(H)$. Two vertex-disjoint subgraphs of $G$ are said \emph{adjacent} if there is an edge between them. We use non-$H$ vertex (edge) to refer to a vertex (edge) that is not in $V(H)$ ($E(H)$). 

Let $C$ be a cycle of $G$. By Jordan Curve Theorem, the image of $C$ separates the plane into two regions called an \emph{external region} and an \emph{internal region}. The external region, denoted by $ext(C)$, is the infinite region of the plane and the internal region, denoted by $int(C)$, is the finite region of the plane. We say a vertex or an edge is embedded \emph{inside} (\emph{outside}) a cycle $C$ if its image belongs to $int(C)$ ($ext(C)$).

%\begin{equation*}
%\Phi = \frac{25n - 5 m -5 p - 2 q}{27} 
%\end{equation*}
\section{Proof of Theorem~\ref{thm:main}} \label{sec:main-thm-proof}

\begin{lemma}\label{lm:cublike-forest}
If $H$ is a $Q_3^{3-}$ subgraph of a planar graph $G$, then any forest $F$ in $G\setminus H$ can be extended to an induced forest of $G$ of order $|F| + 5$.
\end{lemma}
\begin{proof}

Let $a,b,c \in \{u_1,u_2,\ldots,u_8\}$ be three highest-degree vertices of $H$ in $G$. If $a,b$ and $c$ are pairwise non-adjacent. By symmetry of $Q_3$, we can assume~w.l.o.g that $a,b,c$ are $u_1,u_3,u_6$, respectively. Then, $F\cup \{u_2,u_4,u_5,u_7,u_8\}$ is an induced forest in $G$.  Thus, we can suppose that two vertices, say $a,b$, are adjacent. We consider two cases:
\begin{description}
\item[Case 1]Three vertices $a,b,c$ induce a connected subgraph of $H$. Then, there is a face in any planar embedding of $Q_3$ that contains all $a,b$ and $c$.  By symmetry of $Q_3$, we can assume that $a,b,c$ are $u_1,u_2,u_3$, respectively. Since $\deg_G(H)\leq 3$, at least one vertex in $\{u_1,u_3\}$ is a $4^{-}$-vertex of $G$. Let $x$ be a $4^{-}$-vertex in $\{u_1,u_3\}$. Then, $F\cup \{x,u_4,u_5,u_6,u_7\}$ is an induced forest in $G$.
\item[Case 2] Three vertices $a,b,c$ induce a dis-connected subgraph of $H$. By symmetry of $Q_3$, we can assume that $a,b,c$ are $u_1,u_2,u_7$, respectively. Then, $F\cup \{u_3,u_4, u_5,u_6,u_8\}$ is an induced forest in $G$.
\end{description}
\end{proof}

\begin{lemma}\label{lm:T6-like-forest}
If $K$ is a $T^{3-}_6$ subgraph of $G$, then any forest $F$ in $G\setminus K$ can be extended to an induced forest of $G$ of order $|F| + 4$.
\end{lemma}
\begin{proof}
By symmetry of $T_6$, we can assume~w.l.o.g that cycle $C = v_1v_2v_5v_3$ has the highest degree among cycles inducing faces of $H$. Let $X = \{v_1,v_6,v_4,v_5\}$. Suppose $K$ has a between vertex, say $v$, that has at least two non-$K$ edges in $G$. By the degree assumption of $C$, $v$ must be a vertex in $C$. If $v \not\in X$, then $F\cup X$ is an induced forest of $G$ of order $|F|+4$. If $v\in X$, then $F\cup \{v_2,v_4,v_6,v_3\}$ is an induced forest of $G$. 

Thus, we can assume that every vertex of $K$ has at most one non-$K$ edge. If at most one vertex in $X$ is a between vertex of $K$, then $F\cup X$ is an induced forest of $G$. Thus, we can assume that at least two vertices in $X$ are between. Since $\deg_G(K) \leq 3$, at most one of two vertices $v_2$ and $v_3$ is a between vertex. Let $x$ be the non-between vertex in $\{v_2,v_3\}$. By the degree assumption of $C$, at most one vertex among $\{v_4,v_6\}$ is between. We have two cases:
\begin{description}
\item[Case 1] No vertex in $\{v_4,v_6\}$ is between. Then, $F\cup\{v_2,v_4,v_6,v_3\}$ is an induced forest of $G$.
\item[Case 2] Exactly one vertex in $\{v_4,v_6\}$ is between. Let $y$ be the non-between vertex in $\{v_4,v_6\}$. If both $v_1$ and $v_5$ are between, then $v_2$ and $v_3$ have no non-$K$ edge since $\deg_G(K) \leq 3$. Thus,  $F\cup\{v_2,v_4,v_6,v_3\}$ is an induced forest of $G$. If $v_1$ is between and $v_5$ is non-between, then $F\cup\{v_5,v_4,v_6,x\}$ is an induced forest of $G$. Otherwise, $v_5$ is between and $v_1$ is non-between. Then, $F\cup\{v_1,x,y,v_5\}$ is an induced forest of $G$.
\end{description}
\end{proof}

\begin{observation} \label{ob:Q-3-dis}
 Any two $Q^{2-}_3$ subgraphs of $G$ must be vertex-disjoint.
\end{observation}
\begin{proof}
We observe that any non-trivial cut of $Q_3$ has at least 3 edges. Let $H$ and $K$ be two $Q^{2-}_3$ subgraphs of $G$ that share a subset of vertices $X$. Then, the cut $(V(H)\setminus X, X)$ has at least 3 edges. Thus, $\deg_G(K)\geq 3$, contradicting that $K$ is a $Q_3^{2-}$ subgraph.
\end{proof}

\paragraph{Proof of Theorem~\ref{thm:main}} Let $\rho(G) = p(G) + q(G) + n(G)$. We prove Theorem~\ref{thm:main} by induction on $\rho(G)$. The base case is when $\rho(G) = 0$, Theorem~\ref{thm:main} trivially holds. We consider three cases:

\begin{description}
\item[Case 1] Graph $G$ has no $Q_3^{-1}$ subgraph or $T_6$ component. Then, $p(G) + q(G) = 0$. Using a linear programming solver \footnote{We use lp\_solve package \href{http://lpsolve.sourceforge.net/5.5/index.htm}{http://lpsolve.sourceforge.net/5.5/index.htm}. The full implementation can be found at the author's homepage \url{http://web.engr.oregonstate.edu/~lehu/res/lp_final.lp}} to solve Linear Program~\ref{eq:our-linear}, we found that $a - 2b$ is maximized when $a = \frac{25}{27},b = c = \frac{5}{27}, d = \frac{2}{27}$. Corollary~\ref{co:sub-main} implies that if $G$ has an induced forest $F$ of order at least $\frac{5n(G)}{9}$.

\item[Case 2] Graph $G$ contains a $T_6$ component, then $p(G\setminus T_6) \leq p(G)$ and  $q(G\setminus T_6) < q(G)$. Thus, $\rho(G\setminus T_6) < \rho(G)$. By induction, $\varphi(G\setminus T_6)\geq \frac{5n(G\setminus T_6)}{9}  = \frac{5(n(G)-6)}{9} $. By Lemma~\ref{lm:T6-like-forest}, we can collect 4 verties from $T_6$. That implies:
\begin{equation*}
\varphi(G) \geq \varphi(G\setminus T_6) + 4\geq \frac{5(n(G)-6)}{9}  + 4 > \frac{5n(G)}{9} 
\end{equation*}

\item[Case 3] Graph $G$ contains a $Q^{1-}_6$ subgraph, say $H$. Since $H$ has degree at most 1 in $G$, removing $H$ from $G$ can create at most one $T_6$ component and at most one new $Q_3^{1-}$ subgraph. Thus, $p(G\setminus H) \leq p(G)$ and $q(G\setminus H) \leq q(G)+1$. Since $n(G\setminus H) \leq n(G) - 8$, we have $\rho(G\setminus H) < \rho(G)$. By induction, we have $\varphi(G\setminus H) \geq \frac{5\left( n(G)- 8\right)}{9}$. By Lemma~\ref{lm:cublike-forest}, we can collect $5$ vertices from $H$. That implies:
\begin{equation*}
\varphi(G) \geq \varphi(G\setminus H) + 5 \geq \frac{5(n(G)-8)}{9}  + 5>  \frac{5n(G)}{9} 
\end{equation*}
 \end{description}

\section{Proof of Theorem~\ref{thm:sub-main}} \label{sec:proof-sub-main}

Let $G$ be a counter-example of minimal order.  We begin our proof with Observation~\ref{obs:meta-exclusion}, that we will frequently make use of in deriving contradiction. 
\begin{observation} \label{obs:meta-exclusion} Let $L$ be a subgraph of $G$. Let $\alpha, \beta, \gamma,\eta$ be such that: 
\begin{equation}
\begin{split}
\alpha &= n(G) - n(G\setminus L)\\
\beta &\leq m(G) - m(G\setminus L)\\
\gamma &\leq p(G) - p(G\setminus L)\\
\eta &\leq q(G) -q(G\setminus L)
\end{split}
\end{equation}
If we can collect $\lambda$ vertices from $L$, then, $\lambda - \alpha a + \beta b + c\gamma + d\eta$ must be negative.
\end{observation}
\begin{proof}
Suppose that $\lambda - \alpha a + \beta b + c\gamma + d\eta$ is non-negative. Since $G$ is a minimal counter-example, $G\setminus L$ has an induced forest of order at least $a n(G\setminus L) - b m(G\setminus L) - c p(G\setminus L) - dq(G\setminus L)$ which is at least:
\begin{equation*}
an(G) - bm(G) - cp(G) - dq(G) + \beta b + c\gamma + d\eta -\alpha a.
\end{equation*} 
By collecting $\lambda$ vertices from $L$, we get a forest in $G$ of order at least:
\begin{equation*}
an(G) - bm(G) - cp(G) - dq(G) + \lambda + \beta b + c\gamma + d\eta -\alpha a
\end{equation*} 
Since $\lambda - \alpha a + \beta b + c\gamma + d\eta$ is non-negative, $\varphi(G) \geq an(G) - bm(G) - cp(G) - dq(G)$, contradicting that $G$ is a counter-example. 
\end{proof}

\paragraph{Overview of the proof} Our proof of Theorem~\ref{thm:sub-main} relies on the following structural theorem that was proved by Salavatipour~\cite{Salavatipour06}.

\begin{theorem}\label{thm:structure}
If $G$ is a two-edge connected triangle-free planar graph, then, $G$ contains (1) a $2^-$-vertex, or (2) a $4$-face with at least one $3$-vertex, or (3) a $5$-face with at least four $3$-vertices.
\end{theorem}
 At high level, we build a linear program, called $\lp$, that initially contains trivial constraints~\eqref{seq:our-iso}, \eqref{seq:our-iso1},  \eqref{seq:our-triv}, \eqref{seq:our-triv1} and~\eqref{seq:our-triv2}. We then consider a finite set of subgraphs, say $\mathcal{L}$, that a triangle-free planar graph can have. For each subgraph, say $H$, in $\mathcal{L}$, by removing it from $G$, we reduce the number of vertices and edges of $G$ by at least, say, $\alpha$ and $\beta$, respectively. Then, we show that we can add $\gamma$ vertices from $H$ to a large induced forest of $G\setminus H$ to get an induced forest of $G$. Observation~\ref{obs:meta-exclusion} tells us that if we choose $a, b,c$ and $d$ such that $\lambda - \alpha a + \beta b + c\gamma + d\eta \geq 0$, then $G$ cannot be a counter-example. Thus, a counter-example graph $G$ cannot contain the subgraph $H$. In other words, by adding the constraint $\lambda - \alpha a + \beta b + c\gamma + d\eta \geq 0$ to $\lp$, we \emph{exclude} $H$ from $G$. We repeat this argument for every subgraph in $\mathcal{L}$ and keep adding linear constraints along the way to $\lp$. Finally, we get a linear program represented by $\lp$ and we show that $\lp$ is equivalent to Linear Program~\ref{eq:our-linear} by removing redundant constraints from $\lp$. Thus, by choosing $a,b,c$ and $d$ satisfies Linear Program~\ref{eq:our-linear}, the counter-example $G$ does not exist, thereby, proving Theorem~\ref{thm:sub-main}. 
 
 In Subsection~\ref{subsec:ex-2-vertex}, we prove that $G$ is two-edge connected and $\delta(G) \geq 3$. In Subsection~\ref{subsec:ex-4face-3vers}, we prove that $G$ has no 4-face with at least one 3-vertex. In Subsection~\ref{subsec:ex-5face-4-3vers}, we prove that $G$ has no 5-face with at least four 3-vertices. This is a contradiction by Theorem~\ref{thm:structure}.

% At high level, we build a linear program, called $\lp$, that initially contains trivial constraints~\eqref{seq:our-iso},~\eqref{seq:our-iso1},~\eqref{seq:our-triv},~\eqref{seq:our-triv1} and ~\eqref{seq:our-triv2}. We consider all possible configurations that $G$ can have in a particular order. For each configuration, we can show that $G$ has induced forest of size at least $an(G) - bm(G) - cp(G) - dq(G)$ if variables $a,b,c,d$ satisfy a linear constraint. Thus, to exclude the configuration from $G$, we add the linear constraint to $\lp$. Finally, we get a linear program represented by $\lp$ and we show that $\lp$ is equivalent to Linear Program~\ref{eq:our-linear} by removing redundant constraints from $\lp$. 
%TO-DO: Say something like applying Lemma~\ref{lm:meta-exclusion} in this form to get a contradiction. In Section~\ref{sec:ex-cub-T-6-like}, we prove that $G$ contains no cube-like and $T_6$-like subgraph. In Section~\ref{sec:ex-2-vertex}, we prove that $\delta(G) \geq 3$. In Section abc, we prove xyz. This is a contradiction by Theorem~\ref{thm:structure}.

\subsection{Excluding $Q_3^d$ and $T_6^d$ subgraphs} \label{sec:ex-cub-T-6-like}

In this section, by adding more constraints to $\lp$, we will prove that the minimal counter example $G$ cannot contain any $Q_3^d$ or $T_6^d$ subgraph for $d\leq 5$ if $\lp$ is satisfied. 
 
\begin{claim} \label{clm:exclue-Q-3-0}
Graph $G$ has no $Q_3$ component.  
\end{claim}
\begin{proof}

Let $H$ be a $Q_3$ component of $G$. By Lemma~\ref{lm:cublike-forest}, we can collect 5 vertices from $H$. Since $Q_3$ has $8$ vertices, 12 edges, by Observation~\ref{obs:meta-exclusion} with $L = Q_3$ and  $(\alpha,\beta, \gamma,\eta,\lambda) = (8,12,1,0,5)$, $5-8a + 12 b + c $ must be negative. Thus, we obtain contradiction by adding Inequality~\eqref{lpc:Q-3} to $\lp$.
\begin{equation}\label{lpc:Q-3}
5-8a + 12 b + c \geq  0
\end{equation}
\end{proof}

\begin{claim} \label{clm:exclue-T-6}
 Graph $G$ has no $T_6$ component.  
\end{claim}
\begin{proof}
Let $H$ is a $T_6$ component of $G$. By Lemma~\ref{lm:T6-like-forest}, we can collect 4 vertices from $H$. By Observation~\ref{obs:meta-exclusion} with $L = T_6$ and $(\alpha,\beta,\gamma,\eta,\lambda) = (6,8,0,1,4)$, $4-6a + 8b + d$ must be negative. Thus, we obtain contradiction by adding Inequality~\eqref{lpc:T-6} to $\lp$.
\begin{equation}\label{lpc:T-6}
4-6a + 8b + d \geq 0 
\end{equation}
\end{proof}

\noindent Claim~\ref{clm:exclue-T-6} implies that if $\lp$ is satisfied, the counter-example $G$ has no $T_6$ component. 
\begin{claim} \label{clm:exclue-Q-3-1}
Graph $G$ excludes $Q_3^{1-}$ as a subgraph. 
\end{claim}
\begin{proof}
By Claim~\ref{clm:exclue-Q-3-0}, we only need to exclude $Q_3^{1}$ from $G$. Let $H$ be a $Q_3^1$ subgraph of $G$. Let $G' = G \setminus H$. If $H$ is adjacent to a $Q_3^2$ subgraph of $G$, then $p(G') = p(G)$ and $q(G') = q(G) = 0$.  By Lemma~\ref{lm:cublike-forest}, we can collect 5 vertices from $H$. By applying Observation~\ref{obs:meta-exclusion} with $L = H$ and  $(\alpha,\beta,\gamma,\eta,\lambda) = (8,13,0,0,5)$, $5-8a + 13 b$  must be negative. Thus, we obtain contradiction by adding Inequality~\eqref{lpc:Q-3-1a} to $\lp$.
\begin{equation}\label{lpc:Q-3-1a}
5-8a + 13 b  \geq 0
\end{equation} 
If $H$ is not adjacent to a $Q_3^2$ subgraph, then $p(G') = p(G)-1$. Note that $G'$ can has a $T_6$ component if $H$ is adjacent to a $T_6^1$ subgraph in $G$. By Observation~\ref{obs:meta-exclusion} with $L = H$ and $(\alpha,\beta,\gamma,\eta,\lambda) = (8,13,1,-1,5)$, $5-8a + 13 b + c - d$ must be negative. Thus, we obtain contradiction by adding Inequality~\ref{lpc:Q-3-1b} to $\lp$.

\begin{equation}\label{lpc:Q-3-1b}
5-8a + 13 b + c - d  \geq 0
\end{equation} 
\end{proof}

\noindent Claim~\ref{clm:exclue-Q-3-0} and~\ref{clm:exclue-Q-3-1} imply that if $\lp$ is satisfied, $G$ has no $Q^{1-}_3$ subgraph. Herein, we can assume that the counter-example graph $G$ has $p(G) = q(G) = 0$.

\begin{claim} \label{clm:exclue-T-6-1}
Graph $G$ excludes $T_6^{1-}$ as a subgraph. 
\end{claim}
\begin{proof}
By Claim~\ref{clm:exclue-T-6}, we only need to exclude $T_6^{1}$ from $G$.
Let $K$ be a $T_6^{1}$ subgraph of $G$. Let $H_1,\ldots, H_t$ be the subgraphs of $G$ such that $H_j$ is a $Q_3^1$ subgraph of $G\setminus \{K\cup \{H_1,\ldots,H_{j-1}\}\}$ and $H_j$ is adjacent to $H_{j-1}$ in $G$. Let $t$ be the maximum index such that $G\setminus \{K\cup H_1\cup \ldots \cup H_{t}\}\}$ contains no $Q^1_3$ subgraph. It may be that none of $H_j$ exists and we define $t = 0$ in this case. Let $KH  = K\cup \{H_1,\ldots,H_{t}\}$. We have $\deg_G(KH) = 1$. Thus, $G\setminus KH$ cannot contain any $Q_3$ component, since otherwise, it would be $Q_3^1$ in $G$, contradicting Claim~\ref{clm:exclue-Q-3-1}. Since $\deg_G(KH) = 1$, $G\setminus KH$ contains at most one $T_6$ component. By Lemma~\ref{lm:cublike-forest} and Lemma~\ref{lm:T6-like-forest}, we can collect $5t + 4$ vertices from $KH$. By Observation~\ref{obs:meta-exclusion} with $L = KH$ and $(\alpha,\beta,\gamma,\eta,\lambda) = (8t+6,13t+9,0,-1,5t+4)$, $(5t+4)-(8t+6)a + (13t+9)b - d$ must be negative.  Thus, we obtain contradiction by adding Inequality~\eqref{lpc:T-6-1} to $\lp$.
\begin{equation} \label{lpc:T-6-1}
(5t+4)-(8t+6)a + (13t+9)b - d \geq 0
\end{equation}   
\end{proof}

\begin{claim} \label{clm:exclue-Q-3-2}
Graph $G$ excludes $Q_3^{2-}$ as a subgraph. 
\end{claim}
\begin{proof}By Claim~\ref{clm:exclue-Q-3-1}, we only need to exclude $Q_3^{2}$ from $G$. Let $H$ be a $Q_3^2$ subgraph of $G$. Suppose that $G\setminus H$ contains a $T_6$ component, say $K$. By Claim~\ref{clm:exclue-T-6-1}, $K$ is the only $T_6$ component of $G\setminus H$. By Claim~\ref{clm:exclue-Q-3-1}, $p(G\setminus H)= 0$. By  Observation~\ref{obs:meta-exclusion} with $L = H$ and $(\alpha,\beta,\gamma,\eta,\lambda) = (8,14,0,-1,5)$, $5-8a + 14b - d $ must be negative.  Thus, we obtain contradiction by adding Inequality~\eqref{lpc:Q-3-2a} to $\lp$.
\begin{equation}\label{lpc:Q-3-2a}
5-8a + 14b - d \geq 0
\end{equation} 

Thus, we may assume that $G\setminus Q_3^2$ has no $T_6$ component for any $Q_3^2$ subgraph of $G$. Without loss of generality, we choose $H$ to be a $Q_3^2$ subgraph such that $G\setminus H$ has the least number of $Q^{-1}_3$ subgraphs. By Claim~\ref{clm:exclue-Q-3-1}, $G\setminus H$ has at most two $Q^{-1}_3$ subgraphs. If $G\setminus H$ has exactly one $Q^{1-}_3$ subgraph, say $M$, then $M$  must be adjacent to $H$. By Observation~\ref{obs:meta-exclusion} with $L = H$ and $(\alpha,\beta,\gamma,\eta,\lambda) = (8,14,-1,0,5)$, $5-8a + 14b - c $ must be negative.  Thus, we obtain contradiction by adding Inequality~\eqref{lpc:Q-3-2b} to $\lp$.
\begin{equation}\label{lpc:Q-3-2b}
5-8a + 14b - c \geq 0
\end{equation} 

If $G\setminus H$ has two $Q^{1-}_3$ subgraphs. By Claim~\ref{clm:exclue-Q-3-1}, two $Q_3^{1-}$ subgraphs are $Q_3^{1}$ subgraphs. By our choice of $H$, we conclude that, for any $Q_3^2$ subgraph of $G$, $G\setminus Q_3^2$ must have exactly two $Q^1_3$ subgraphs. Since $G$ excludes $Q_3^1$ by Claim~\ref{clm:exclue-Q-3-1}, any $Q_3^2$ subgraph of $G$ must adjacent to two other $Q_3^2$ subgraphs. Let $\mathcal{H}$ be a graph such that each vertex of $\mathcal{H}$ corresponds to a $Q^2_3$ subgraph of $G$ and each edge of $\mathcal{H}$ connects two adjacent $Q^2_3$ subgraphs of $G$. Then,  $\mathcal{H}$ is a 2-regular graph. In other words, $\mathcal{H}$ is a collection of cycles. By Lemma~\ref{lm:cublike-forest}, we can collect $5|V(\mathcal{H})|$ vertices from $Q_3^2$ subgraphs of $G$. By Observation~\ref{obs:meta-exclusion} with $L$ to be the induced subgraph of $G$ induced by vertices in all $Q_3^2$ subgraphs of $G$ and $(\alpha,\beta,\gamma,\eta,\lambda) = (8|V(\mathcal{H})|,13|V(\mathcal{H})|,0,0,5|V(\mathcal{H})|)$, $|V(\mathcal{H})|(5-8a + 13b)$ must be negative, this contradicts Inequality~\eqref{lpc:Q-3-1a}.
\end{proof}

\begin{claim} \label{clm:exclue-Q-3-3}
Graph $G$ excludes $Q_3^{3-}$ as a subgraph.
\end{claim}
\begin{proof}
By Claim~\ref{clm:exclue-Q-3-2}, we only need to exclude $Q_3^{3}$ from $G$.
Let $H$ be a $Q_3^3$ subgraph in $G$. By Claim~\ref{clm:exclue-Q-3-2}, $G\setminus H$ contains at most one $Q^{1-}_3$ subgraph. By Claim~\ref{clm:exclue-T-6-1}, $G\setminus H$ contains at most one $T_6$ component. By Observation~\ref{obs:meta-exclusion} with $L = H$ and $(\alpha,\beta,\gamma,\eta,\lambda) = (8,15,-1,-1,5)$, $5-8a + 15b - c - d $ must be negative. Thus, we obtain contradiction by adding Inequality~\eqref{lpc:Q-3-3} to $\lp$.
\begin{equation} \label{lpc:Q-3-3}
5-8a + 15b - c - d \geq 0
\end{equation} 
\end{proof}

\noindent We obtain the following corollary of Claim~\ref{clm:exclue-Q-3-3}.
\begin{corollary} \label{cor:no-cub-deg-2-comp}
If $H$ is a subgraph of degree 2 of $G$ and $\lp$ is satisfied, then $G\setminus H$ has no $Q^{1-}_3$ subgraph.
\end{corollary}

\begin{claim}\label{clm:exclude-T-6-2}
Graph $G$ excludes $T_6^{2-}$ as a subgraph.
\end{claim}
\begin{proof}
By Claim~\ref{clm:exclue-T-6-1}, we only need to exclude $T_6^{2}$ from $G$. Let $H$ be a $T_6^{2}$ subgraph of $G$.  By Corollary~\ref{cor:no-cub-deg-2-comp}, $G\setminus H$ has no $Q^{1-}_3$ subgraph. By Claim~\ref{clm:exclue-T-6-1}, $G\setminus H$ has at most one $T_6$ component. By Observation~\ref{obs:meta-exclusion} with $L = H$ and $(\alpha,\beta,\gamma,\eta,\lambda) = (6,10,0,-1,4)$, $4-6a + 10b - d$ must be negative. Thus, we obtain contradiction by adding Inequality~\eqref{lpc:T-6-2} to $\lp$.
\begin{equation}\label{lpc:T-6-2}
4-6a + 10b - d \geq 0
\end{equation}
\end{proof}
\begin{claim} \label{clm:no-deg-5}
Graph $G$ has no $5^+$-vertex.
\end{claim}
\begin{proof}
Let $v$ be a $5^+$ vertex in $G$ and $G' = G-\{v\}$. Suppose that $G'$ has a $Q^{1-}_3$ subgraph $H$. By planarity, $v$ must be embedded in one face of $H$. Since faces of $H$ has length $4$ and $G$ is triangle-free, $v$ has at most two neighbors in $H$. That implies $H$ is a $Q_3^{3-}$ subgraph of $G$, contradicting Claim~\ref{clm:exclue-Q-3-3}. Thus $p(G') = 0$.  Suppose that $G'$ has a $T_6$ component $K$. By planarity, $v$ must be embedded in one face of $K$. Since $G$ is triangle-free, $v$ has at most two neighbors in $K$. That implies $K$ is $T_6^{2-}$, contradicting Claim~\ref{clm:exclude-T-6-2}. Thus $q(G') = 0$.  By Observation~\ref{obs:meta-exclusion} with $L = v$ and $(\alpha,\beta,\gamma,\eta,\lambda) = (1,5,0,0,0)$, $5b-a$ must be negative. Thus, we obtain contradiction by adding Inequality~\eqref{lpc:deg-5} to $\lp$.
\begin{equation}\label{lpc:deg-5}
5b-a \geq 0
\end{equation}
\end{proof}
\begin{lemma}\label{lm:Q-3-4-forest}
If $H$ is $Q_3^{5-}$ subgraph of $G$ and every vertex of $H$ has degree at most 4 in $G$, then any forest $F$ in $G\setminus H$ can be extended to a forest of $G$ of order $|F| + 5$.
\end{lemma}
\begin{proof}
By Lemma~\ref{lm:cublike-forest}, we can assume that $H$ is $Q_3^{4}$ or $Q_3^{5}$.  By symmetry of $Q_3$, we can choose an embedding of $G$ such that the inner-most face $u_1u_2u_3u_4$, denoted by $f$, of $H$ has the most number of 3-vertices. We have three cases:

\begin{description}
\item[Case 1] Face $f$ has at least three 3-vertices, say $u_1,u_2,u_3$, then $F\cup\{u_1,u_2,u_3,u_6,u_8\}$ is an induced forest of $G$. 
\item[Case 2] Face  $f$ has only one 3-vertex, say $u_1$, then every face that contains $u_1$ on the boundary must has at least three 4-vertices by the choice of the inner-most face of $H$. That implies $\deg_G(H)\geq 6$; a contradiction.
\item[Case 3] Face $f$ has exactly two 3-vertices. By the choice of $f$, every face of $H$ has at most two 3-vertices. Suppose that $H$ has two adjacent 3-vertices. By symmetry of $Q_3$, we can choose an embedding of $G$ such that two 3-vertices of $f$ are adjacent. We can assume~w.l.o.g they are $u_1$ and $u_2$. Thus,  $u_5$ and $u_6$ must be 4-vertices. Since $\deg_G(H)\leq 5$, at most one vertex in $\{u_7,u_8\}$ is a 4-vertex. Let $u^*$ be a 3-vertex in $\{u_7,u_8\}$. Since non-$H$ edges of $u_6$ and $u_4$ are embedded in different faces of $G$, $F\cup \{u_6,u_1,u_{2},u_{4}, u^*\}$ is an induced forest of $G$. If $H$ has no two adjacent 3-vertices, we can assume~w.l.o.g that $u_1$ and $u_3$ are two 3-vertices of $H$. Thus, $u_2,u_4,u_5,u_7$ are 4-vertices. Since $\deg_G(H) \leq 5$, at least one vertex in $\{u_6,u_8\}$ is a 3-vertex. We define $u^*$ to be $u_2$ if $u_8$ is a 4-vertex and $u^* = u_4$ if $u_6$ is a 4-vertex. Then, $F\cup\{u_1,u_3,u_6,u_8,u^*\}$ is an induced-forest in $G$. \qedhere
\end{description} 
\end{proof}

\begin{claim} \label{clm:exclue-Q-3-4}
Graph $G$ excludes $Q_3^{4-}$ as a subgraph.
\end{claim}
\begin{proof}

By Claim~\ref{clm:exclue-Q-3-3}, we only need to exclude $Q_3^{4}$ from $G$. Let $H$ be a $Q_3^{4}$ subgraph of $G$. By Claim~\ref{clm:no-deg-5}, between vertices of $H$  are $4$-vertices.  By Claim~\ref{clm:exclue-Q-3-3}, $G\setminus H$ has at most one $Q^{1-}_3$ subgraph. By Claim~\ref{clm:exclude-T-6-2}, $G\setminus H$ has at most one $T_6$ component. By Lemma~\ref{lm:Q-3-4-forest}, we can collect 5 vertices from $H$. By Observation~\ref{obs:meta-exclusion} with $L = H$ and $(\alpha,\beta,\gamma,\eta,\lambda) = (8,16,-1,-1,5)$, $5 - 8a + 16b - c - d$ must be negative.  Thus, we obtain contradiction by adding Inequality~\eqref{lpc:Q-3-4} to $\lp$.
\begin{equation} \label{lpc:Q-3-4}
5 - 8a + 16b - c - d  \geq 0
\end{equation}   
\end{proof}

\begin{claim}\label{clm:exclude-T-6-3}
Graph $G$ excludes $T_6^{3^-}$ as a subgraph. 
\end{claim}
\begin{proof}
By Claim~\ref{clm:exclude-T-6-2}, we only need to exclude $T_6^{3}$ from $G$. Let $H$ be a $T_6^3$ subgraph of $G$. By Claim~\ref{clm:exclue-Q-3-4}, $G\setminus H$ has no $Q_3$-like subgraph and by Claim~\ref{clm:exclude-T-6-2},  $G\setminus H$ has at most one $T_6$ component. By Lemma~\ref{lm:T6-like-forest}, we can collect $4$ vertices from $H$. By Observation~\ref{obs:meta-exclusion} with $L = H$ and $(\alpha,\beta,\gamma,\eta,\lambda) = (6,11,0,-1,4)$, $4 - 6a + 11b - d$ must be negative.  Thus, we obtain contradiction by adding Inequality~\eqref{lpc:T-6-3} to $\lp$.
\begin{equation} \label{lpc:T-6-3}
4 - 6a + 11b - d  \geq 0
\end{equation}   
\end{proof}

\begin{claim}\label{clm:exclude-T-6-4-5}
Graph $G$ excludes any $T_6^{5-}$ subgraph that has all between vertices on the same face.
\end{claim}
\begin{proof}
Suppose that $G$ contains a  $T_6^{5-}$ subgraph $H$ as in the claim. By Claim~\ref{clm:exclude-T-6-3}, $\deg_G(H) \geq 4$. By symmetry of $H$, we can assume~w.l.o.g that the outer face $v_1v_2v_5v_3$ of $H$ contains all between vertices. Let $K$ be the subgraph of $G$ induced by $\{v_1,v_2,v_3,v_4,v_6\}$. By Claim~\ref{clm:no-deg-5}, $v_1$ has at most one non-$H$ incident edge. Thus, we can collect $\{v_1, v_4,v_6\}$ from $K$. Since $\deg_G(K) \leq 5$, by Claim~\ref{clm:exclue-Q-3-4}, $G\setminus K$ has at most one $Q^{1-}_3$ subgraph. If $G\setminus K$ has exactly one $Q^{1-}_3$ subgraph, the $Q^{1-}_3$ subgraph in $G\setminus K$ must has three edges to $K$ in $G$. That implies $G\setminus K$ has no $T_6$ component, by Claim~\ref{clm:exclude-T-6-3}. Since $\deg_G(H)$ is at least $4$, $m(G) - m(G\setminus K) \geq 10$. By Observation~\ref{obs:meta-exclusion} with $L = K$ and $(\alpha,\beta,\gamma,\eta,\lambda) = (5,10,0,-1,3)$, $3 - 5a + 10b  - c $ must be negative.  Thus, we obtain contradiction by adding Inequality~\eqref{lpc:T-6-4-5a} to $\lp$.
\begin{equation} \label{lpc:T-6-4-5a}
3 - 5a + 10b  - c  \geq 0
\end{equation}

If $G\setminus K$  has no $Q^{1-}_3$ subgraph, by Claim~\ref{clm:exclude-T-6-3}, $G\setminus K$ has at most one $T_6$ component. By Observation~\ref{obs:meta-exclusion} with $L = K$ and $(\alpha,\beta,\gamma,\eta,\lambda) = (5,10,-1,0,3)$, $3 - 5a + 10b  - d $ must be negative.  Thus, we obtain contradiction by adding Inequality~\eqref{lpc:T-6-4-5b} to $\lp$.
\begin{equation} \label{lpc:T-6-4-5b}
3 - 5a + 10b  - d  \geq 0
\end{equation}

\end{proof}

\begin{claim} \label{clm:exclue-Q-3-5}
Graph $G$ excludes $Q_3^{5^-}$ as a subgraph.
\end{claim}
\begin{proof}
By Claim~\ref{clm:exclue-Q-3-4}, we only need to exclude $Q_3^{5}$ from $G$. Let $H$ be a $Q_3^{5}$ subgraph of $G$. By Claim~\ref{clm:no-deg-5}, between vertices of $H$ has degree exactly $4$. By Claim~\ref{clm:exclue-Q-3-4}, $G\setminus H$ has at most one $Q^{1-}_3$ subgraph. By Claim~\ref{clm:exclude-T-6-3}, $G\setminus H$ has at most one $T_6$ component. By Lemma~\ref{lm:Q-3-4-forest}, we can collect 5 vertices from $H$. By Observation~\ref{obs:meta-exclusion} with $L = K$ and $(\alpha,\beta,\gamma,\eta,\lambda) = (8,17,-1,-1,5)$, $5 - 8a + 17b  - c -d  $ must be negative.  Thus, we obtain contradiction by adding Inequality~\eqref{lpc:Q-3-5} to $\lp$.
\begin{equation} \label{lpc:Q-3-5}
5 - 8a + 17b  - c -d   \geq 0
\end{equation}
\end{proof}

\begin{claim} \label{clm:meta-cub-excl}
If $H$ is a connected subgraph of $G$, then $G\setminus H$ has no $Q^{1-}_3$ subgraph and $T_6$ component.
\end{claim}
\begin{proof}
Suppose that $G\setminus H$ contains a $Q^{1-}_3$ subgraph $K$. Since $H$ is connected, its vertices are embedded in on face of $K$, say the infinite face. Thus, $G$ has at most 4 edges connecting vertices of $H$ and vertices of $K$. Since $K$ has degree at most one in $G\setminus H$, $K$ has degree at most $5$ in $G$, contradicting Claim~\ref{clm:exclue-Q-3-5}. Suppose that $G\setminus H$ contains a $T_6$ component $M$. Since $H$ is connected, their vertices are embedded inside one face of $M$. Thus, there exists one face of $M$ contains all of its between vertices.  Since $G$ only has $4^{-}$-vertices, $\deg_G(M) \leq 5$, contradicting Claim~\ref{clm:exclude-T-6-4-5}.
\end{proof}

\subsection{Excluding low degree vertices} \label{subsec:ex-2-vertex}

As shown in Section~\ref{sec:ex-cub-T-6-like}, if $\mathcal{LP}$ is satisfied, $G$ has  $p(G) = 0$ and $q(G) = 0$. Thus, we only need to prove $\varphi(G) \geq an(G) -bm(G)$ to obtain contradiction.

\begin{claim} \label{clm:2-edge-conn}
$G$ is two-edge connected.
\end{claim}
\begin{proof}
Suppose that the claim fails, then either $G$ is disconnected or $G$ is connected and has a bridge $e$. If $G$ is disconnected, let $G_1$ be any connected component of $G$ and $G_2 = G\setminus G_1$. If $G$ is connected and has a bridge $e$, let $G_1,G_2$ be two components of $G\setminus \{e\}$. Since $\deg_G(G_1) \leq 1$, by Claim~\ref{clm:exclue-Q-3-5}, $p(G_1) = p(G_2) = 0$. By Claim~\ref{clm:exclude-T-6-3}, $q(G_1) = q(G_2) = 0$. Since $G_1,G_2$ has strictly smaller order than $G$, they have two forests $F_1, F_2$ of order at least $a n(G_1) -b m(G_1), a n(G_2) -b m(G_2)$, respectively. Thus, $F_1 \cup F_1$ is an induced forest of $G$ of order at least:

\begin{equation*}
a (n(G_1) + n(G_2)) -b (m(G_1) + m(G_2)) \geq a n(G) - b m(G)
\end{equation*}
This contradicts that $G$ is a counter-example.
\end{proof}

\noindent A direct corollary of Claim~\ref{clm:2-edge-conn} is that $\delta(G) \geq 2$. 

\begin{claim}\label{clm:4-face-2-vertex}
If $v$ is a 2-vertex, then its neighbors must have another common neighbor.
\end{claim}
\begin{proof}
 Let $G'$ be the graph obtained from $G$ by contracting an incident edge of $v$. Suppose that $v$ is the only common neighbor of its neighbors, then, $G'$ is triangle-free. Let $u$ be the vertex obtained after the contraction. Any $Q_3^{1-}$ subgraph and $T_6$ component of $G'$ must contain $u$. Thus, $p(G') + q(G') \leq 1$. Since $G'$ has strictly smaller order than $G$, $G'$ has a forest $F'$ of order at least $a n(G') - bm(G') - c p(G') - dq(G')$. We note that $n'(G) = n(G)-1$ and $m(G') = m(G)-1$. If $p(G') = 1$, $F'\cup \{v\}$ is an induced forest in $G$ of order at least:
\begin{equation*}
1 + a(n(G)-1) - b(m(G)-1) - c = a n(G) - b m(G) + 1- a + b -c
\end{equation*}
Thus, by adding Inequality~\eqref{lpc:4f-2va} to $\lp$, we deduce that $\varphi(G) \geq a n(G) - b m(G) $, contradicts that $G$ is a counter-example.
\begin{equation} \label{lpc:4f-2va}
 1 - a + b -c  \geq 0
\end{equation}
If $q(G') = 1$, $F'\cup \{v\}$ is an induced forest in $G$ of order at least:
\begin{equation*}
1 + a(n(G)-1) - b(m(G)-1) - d = a n(G) - b m(G) + 1 - a + b -d
\end{equation*}  
Thus, by adding Inequality~\eqref{lpc:4f-2vb} to $\lp$, we obtain a contradiction.
\begin{equation} \label{lpc:4f-2vb}
1 - a + b -d  \geq 0
\end{equation}
\end{proof}

\begin{claim}\label{clm:2-4-vertex}
None neighbor of a 2-vertex is a $4$-vertex.
\end{claim}
\begin{proof}
Suppose that a neighbor $u$ of a 2-vertex $v$ is a $4$-vertex. Let $G' = G-\{u,v\}$. By Claim~\ref{clm:meta-cub-excl}, $p(G') = q(G') = 0$. Observe that we can add $v$ to any induced forest of $G'$ to get an induced forest of $G$. By Observation~\ref{obs:meta-exclusion} with $L = uv$ and $(\alpha,\beta,\gamma,\eta,\lambda) = (2,5,0,0,1)$, $1 - 2a + 5b  $ must be negative.  Thus, we obtain contradiction by adding Inequality~\eqref{lpc:2v-4v} to $\lp$.
\begin{equation} \label{lpc:2v-4v}
1 - 2a + 5b   \geq 0
\end{equation}
\end{proof}
\begin{claim}\label{clm:2-2-vertex}
None neighbor of a 2-vertex is a 2-vertex.
\end{claim}
\begin{proof}
Suppose that a neighbor $u$ of a 2-vertex $v$ is a 2-vertex. Let $w$ and $w'$ be other neighbors of $u$ and $v$, respectively. By Claim~\ref{clm:2-4-vertex}, $w$ and $w'$ are $3^-$-vertices.  By Claim~\ref{clm:4-face-2-vertex}, $w$ and $w'$ must be adjacent. If both $w$ and $w'$are $2$-vertices, then $G$ is a cycle of 4 vertices. Since $G$ has a forest of order $3$, by Observation~\ref{obs:meta-exclusion} with $L = G$ and $(\alpha,\beta,\gamma,\eta,\lambda) = (4,4,0,0,3)$, $3 - 4a + 4b $ must be negative.  Thus, we obtain contradiction by adding Inequality~\eqref{lpc:2v-2va} to $\lp$.
\begin{equation} \label{lpc:2v-2va}
3 - 4a + 4b    \geq 0
\end{equation}

Thus, we may assume $w$ has degree exactly $3$. By Claim~\ref{clm:meta-cub-excl},  $G-\{u,v,w\}$ has no 
$Q_3^{1-}$ subgraph or $T_6$ component. Since we can collect $\{u,v\}$, by Observation~\ref{obs:meta-exclusion} with $L = \{u,v,w\}$ and $(\alpha,\beta,\gamma,\eta,\lambda) = (3,5,0,0,2)$, $2 - 3a + 5b $ must be negative.  Thus, we obtain contradiction by adding Inequality~\eqref{lpc:2v-2vb} to $\lp$.
\begin{equation} \label{lpc:2v-2vb}
2 - 3a + 5b    \geq 0
\end{equation}
\end{proof}

\begin{claim}\label{clm:3-one-2-vertex}
Any 3-vertex in $G$ is adjacent to at most one 2-vertex.
\end{claim}
\begin{proof}
Suppose otherwise. Let $w$ be a 3-vertex that is adjacent to two 2-vertices $u$ and $v$. Let $G' = G-\{u,v,w\}$. By Claim~\ref{clm:meta-cub-excl}, $p(G') = q(G') = 0$. Since we can collect $\{u,v\}$, by Observation~\ref{obs:meta-exclusion} with $L = \{u,v,w\}$ and $(\alpha,\beta,\gamma,\eta,\lambda) = (3,5,0,0,2)$, $2 - 3a + 5b $ must be negative, contradicting Inequality~\eqref{lpc:2v-2vb}. 
\end{proof}

\begin{lemma}\label{lm:min-deg-3}
Every vertex of $G$ has degree at least 3.
\end{lemma}
\begin{proof}
Let $w_1$ be a 2-vertex of $G$ with two neigbors $w_2,w_4$. By Claim~\ref{clm:4-face-2-vertex}, $w_2$ and $w_4$ must have another common neighbor, say $w_3$. Let $C$ be the cycle $w_1w_2w_3w_4$. By Claim~\ref{clm:2-2-vertex} and~\ref{clm:2-4-vertex}, $w_2$ and $w_4$ are 3-vertices. Let $u$ be the non-$C$ neighbor of $w_2$. By Claim~\ref{clm:3-one-2-vertex}, $u$ and $w_3$ are a $3^+$-vertices. Since $G$ is triangle free, $u$ cannot be a neighbor of $w_3$. Let $H$ be the induced subgraph of $G$ induced by $\{w_1,w_2,w_3,w_4,u\}$. We can collect 3 vertices $w_4,w_1,w_2$ from $H$. By Claim~\ref{clm:meta-cub-excl}, $p(G\setminus H) = q(G\setminus H) = 0$. If $m(G) - m(G\setminus H)$ is at least $9$, by Observation~\ref{obs:meta-exclusion} with $L = H$ and $(\alpha,\beta,\gamma,\eta,\lambda) = (5,9,0,0,3)$, $3 - 5a + 9b $ must be negative.  Thus, we obtain contradiction by adding Inequality~\eqref{lpc:min-deg-3a} to $\lp$.
\begin{equation} \label{lpc:min-deg-3a}
3 - 5a + 9b  \geq 0
\end{equation}
 Thus, we can assume $m(G) - m(G\setminus H)\leq 8$. That implies $u$ must be a neighbor of $w_4$ and $w_2$ is a 3-vertex (see Figure~\ref{fig:min-deg-3}(a)).  Since $G$ is two connected, the non-$H$ neighbor of $u$ must be embedded in the same side with the non-$H$ neighbor of $w_3$ with respect to the cycle $uw_4w_3w_2$. Let $v$ be the non-$H$ neighbor of $w_3$. Let $K$ be the subgraph of $G$ induced by $\{w_1,w_2,w_3,w_4,u,v\}$. If $u$ and $v$ are adjacent, $K$ is $T_6^{3-}$, contradicting Claim~\ref{clm:exclude-T-6-2}. Thus, $u$ and $v$ are not adjacent and hence,  we can collect $\{u,w_1,w_2,w_3\}$ from $K$. By Observation~\ref{obs:meta-exclusion} with $L = K$ and $(\alpha,\beta,\gamma,\eta,\lambda) = (6,9,0,0,4)$, $4 - 6a + 9b $ must be negative.  Thus, we obtain contradiction by adding Inequality~\eqref{lpc:min-deg-3b} to $\lp$.
\begin{equation} \label{lpc:min-deg-3b}
4 - 6a + 9b    \geq 0
\end{equation}
\end{proof}
\begin{figure}[tbh]
  \centering
   \includegraphics[height=1.2in]{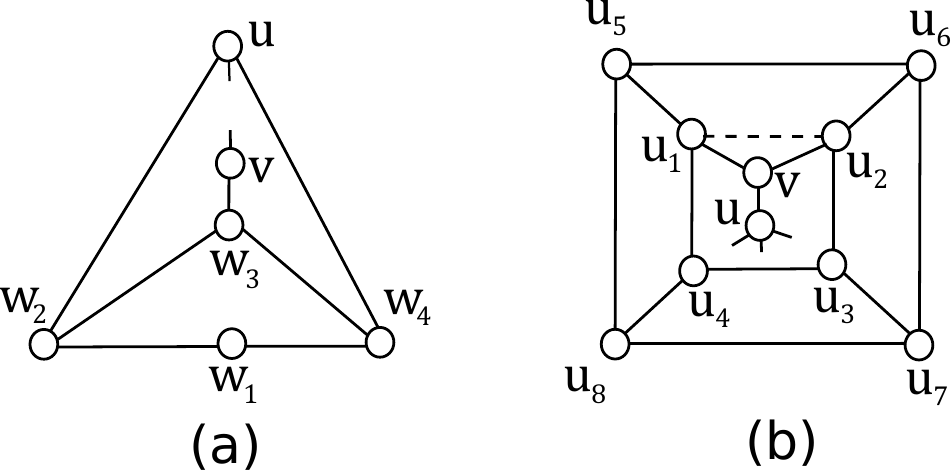}
      \caption{(a) A configuration in the proof of Lemma~\ref{lm:min-deg-3} (b) A configuration in the proof of Claim~\ref{clm:3-4-adj}}
  \label{fig:min-deg-3}
\end{figure} 
\subsection{Avoiding small cut}

A separating cycle is a cycle that separates the plane into two regions, each has non-empty interior. 

\begin{claim} \label{clm:3-4-adj}
Let $v$ be a 3-vertex that is adjacent to a 4-vertex $u$. Then two neighbors of $v$ other than $u$ must share a neighbor other than $v$. 
\end{claim}
\begin{proof}
Suppose otherwise. Let $x,y$ be neighbors of $v$ such that $x\not= u$ and $y\not= u$. Let $G'$ be the graph obtained from $G$ by deleting $u,v$ and adding an edge between $x$ and $y$. We now show that $q(G') = 0$.

By Claim~\ref{clm:meta-cub-excl}, $G-\{u,v\}$ contains no $T_6$ component. If $G'$ contains a $T_6$ component $K$, then $K$ must contain edge $xy$. Since $\delta(G) \geq 3$, $x$ and $y$ must be 3-vertices in $K$. By symmetry of $T_6$, we can assume~w.l.o.g that $x \equiv v_4$ and $y \equiv v_2$. If $u$ is embedded inside the cycle $v_1v_2v_4v_3$, then $v_5$ must be a 2-vertex in $G$. Otherwise, $v_6$ must be a 2-vertex in $G$. Both cases contradict that $\delta(G) \geq 3$.

Suppose that $G'$ contains a $Q_3^{1-}$ component $H$. Edge $xy$ must belongs to $H$. We assume~w.l.o.g that $x\equiv u_1$ and $y\equiv u_2$. Let $M$ be the subgraph of $G$ induced by $\{u_1,u_2,u_3,u_4,u_5,u_6,u_7,u_8,v\}$. By Claim~\ref{clm:meta-cub-excl}, $p(G\setminus M) = q(G\setminus M) = 0$. By the symmetry of $H$, we can assume that $u$ is embedded inside the cycle $u_1vu_2u_3u_4$ (see Figure~\ref{fig:min-deg-3}(b)).  Since $\deg_{G'}(H) \leq 1$,  at most one vertex  in $\{u_1,u_2\}$ is a 4-vertex. Let $z$ be a 3-vertex in $\{u_1,u_2\}$. Since $G$ is triangle-free, $u$ can have at most one neighbor in $\{u_3,u_4\}$. If $u_3$ is a 3-vertex, then we can collect $\{v,z,u_3,u_5,u_7,u_8\}$ from $M$. If $u_4$ is a 3-vertex, then we can collect $\{v,z,u_4,u_6,u_7,u_8\}$ from $M$. Thus, in any case, we can collect $6$ vertices from $M$. By Observation~\ref{obs:meta-exclusion} with $L = M$ and $(\alpha,\beta,\gamma,\eta,\lambda) =(9,14,0,0,6)$, $6 - 9a + 14b $ must be negative.  We obtain contradiction by adding Inequality~\eqref{lpc:3v-4va} to $\lp$.
\begin{equation} \label{lpc:3v-4va}
6 - 9a + 14b  \geq 0
\end{equation}

Thus, we can assume $p(G') = 0$. Hence, $G'$ has a forest $F'$ of order at least $an(G') - bm(G')$. We recall that $xy$ is a non-edge of $G$. Thus, $V(F')\cup \{v\}$ induces a forest of $G$.  Since $n(G') = n(G)-2$ and $m(G') = m(G)-5$, $G$ has a forest on order at least:
\begin{equation*}
a(n(G)-2)- b (m(G)-5) + 1 = an(G) - bm(G) + 1+5b-2a
\end{equation*} 
Thus, we obtain contradiction by adding Inequality~\eqref{lpc:3v-4vb} to $\lp$.
\begin{equation} \label{lpc:3v-4vb}
1 + 5b - 2a  \geq 0
\end{equation}
\end{proof}

\begin{claim} \label{clm:C-non-adj-deg-3}
Let $C$ be a 4-cycle of $G$ that has at least one 3-vertex and at most two 3-vertices. Then, (i) any two 3-vertices of $C$ must be adjacent and two non-$C$ edges adjacent to two 3-vertices must be embedded in the same side of $C$
 and (ii) two non-$C$ edges of a $4$-vertex which is not adjacent to a 3-vertex of $C$ must be embedded in the same side of $C$. 
\end{claim}
\begin{proof}
Let $\{w_1,w_2,w_3,w_4\}$ be clockwise ordered vertices of $C$. Without loss of generality, we assume $w_1$ is a 3-vertex of $C$ and its non-$C$ edge is embedded outside $C$. Suppose that the claim fails. We show that we can collect 2 vertices from $C$. If (i) fails, the other 3-vertex of $C$, denoted by $x$, is $w_3$ or a neighbor of $w_1$ such that its non-$C$ edge is embedded inside $C$. Then, we can collect $\{x,w_1\}$ from $C$. If (ii) fails, let $w_i$ and $w_j$ be two non-adjacent vertices of $C$ such that $w_i$ is a 3-vertex and $w_j$ has two non-$C$ edges that are embedded in different sides of $C$. Then, we can collect $\{w_i,w_j\}$ from $C$. By Claim~\ref{clm:meta-cub-excl}, $p(G\setminus C) = q(G\setminus C) = 0$. Since $m(G\setminus C) \leq m(C) - 10$, by Observation~\ref{obs:meta-exclusion} with $L = C$ and $(\alpha,\beta,\gamma,\eta,\lambda) =(4,10,0,0,2)$, $2 - 4a + 10b $ must be negative, contradicting Inequality~\eqref{lpc:3v-4vb}. 
\end{proof}

\begin{claim} \label{clm:sepC-four-3-vers}
Graph $G$ excludes any separating $4$-cycle that has four 3-vertices.
\end{claim}
\begin{proof}
Let $w_1,w_2,w_3,w_4$ be 3-vertices of a separating $4$-cycle $C$. Since $C$ is separating and $G$ is two-edge connected, two non-$C$ edges of $C$ must be embedded inside $C$ and two other non-$C$ edges must be embedded outside $C$. We assume~w.l.o.g that the non-$C$ edge of $w_1$ is embedded outside $C$. Let $u$ be the non-$C$ neighbor of $w_1$. Let $w_i$, $i\not=1$, be a vertex of $C$ that has its non-$C$ edge embedded outside $C$ and $w_j$ be a vertex of $C$ that has its non-$C$ edge embedded inside $C$. Let $H$ be the subgraph of $G$ induced by $\{w_1,w_2,w_3,w_4,u\}$. We can collect $\{w_1,w_i,w_j\}$ from $H$. We now argue that $m(G\setminus H)\leq m(G) - 10$. Since $G$ is triangle-free, $u$ has at most two neighbors in $C$. If $u$ has only one neighbor in $G$ which is $w_1$, then $m(G\setminus H)\leq m(G) - 10$ since $\delta(G) \geq 3$. If $u$ has exactly two neighbors in $G$, they must be $w_1$ and $w_3$. That means the non-$C$ edge of $w_3$ is embedded outside $C$. Since $C$ is separating, two non-$C$ edges incident to $w_2$ and $w_4$ must be embedded inside $C$. Thus, $u$ must have two non-$H$ incident edges since $G$ is two-edge connected and $\delta(G) \geq 3$. That implies $m(G\setminus H)\leq m(G) - 10$. 

By Observation~\ref{obs:meta-exclusion} with $L = H$ and $(\alpha,\beta,\gamma,\eta,\lambda) =(5,10,0,0,3)$, $3 - 5a + 10b $ must be negative, contradicting Inequality~\eqref{lpc:min-deg-3a} since $b$ is non-negative.
\end{proof}

\begin{claim} \label{clm:sepC-thee-3-vers}
Any separating cycle of length $4$ of $G$ must have at most two 3-vertices. 
\end{claim}
\begin{proof}
Let $C$ be a separating 4-cycle of $G$ that has at least three 3-vertices . By Claim~\ref{clm:sepC-four-3-vers}, $C$ must have exactly three 3-vertices. Let $w_1,w_2,w_3,w_4$ be vertices in the clock-wise order of $C$ such that $w_1,w_2,w_3$ are three 3-vertices. Let $x,y,z$ be the non-$C$ neighbors of $w_1,w_2,w_3$, respectively. Note that $x$ and $z$ may be the same vertex. We assume that $x$ is embedded outside $C$. By Claim~\ref{clm:3-4-adj}, two vertices $w_2$ and $x$ must have a non-$C$ common neighbor and two vertices $w_2$ and $z$ must also have a non-$C$ common neighbor. That implies  $xy$ and $yz$ are edges of $G$. By planarity, $y$ and $z$ must also be embedded outside $C$. Since $C$ is separating and $G$ is two-edge connected, two edges of $w_4$ must be embedded inside $C$. If $x$ and $z$ are the same vertex (see Figure~\ref{fig:sep-4cycle}(a)), then we can collect $\{w_1,w_2,w_3\}$ from the subgraph $H$ that is induced by $\{w_1,w_2,w_3,w_4,x\}$. By Observation~\ref{obs:meta-exclusion} with $L = H$ and $(\alpha,\beta,\gamma,\eta,\lambda) =(5,10,0,0,3)$, $3 - 5a + 10b $ must be negative, contradicting Inequality~\eqref{lpc:min-deg-3a}.
	
Thus, we can assume that $x$ and $z$ are two different vertices (see Figure~\ref{fig:sep-4cycle}(b)). If $x,y,z$ are 3-vetices, then we can collect $\{w_2,w_3,x,y\}$ from the subgraph $K$ of $G$ that is induced by $\{w_1,w_2,w_3, x,y,z\}$. Since $m(G\setminus K) = m(G) - 11$, by Observation~\ref{obs:meta-exclusion} with $L = K$ and $(\alpha,\beta,\gamma,\eta,\lambda) =(6,11,0,0,4)$, $4 - 6a + 11b  $ must be negative, contradicting Inequality~\ref{lpc:min-deg-3b}. Thus, at least one vertex in $\{x,y,z\}$ is a 4-vertex. Let $M$ be the subgraph induced by $\{w_1,w_2,w_3,w_4,x,y,z\}$. Observe that we can collect $\{y,w_1,w_2,w_3\}$ from $M$. Since $m(G\setminus M) \leq m(G) - 14$, by Observation~\ref{obs:meta-exclusion} with $L = M$ and $(\alpha,\beta,\gamma,\eta,\lambda) = (7,14,0,0,4)$, $4 - 7a + 14b $ must be negative.  Thus, we obtain contradiction by adding Inequality~\eqref{lpc:sep4-three-3v} to $\lp$.
\begin{equation} \label{lpc:sep4-three-3v}
4 - 7a + 14b \geq 0
\end{equation} 
\end{proof}
\begin{figure}[tbh]
  \centering
   \includegraphics[height=1.2in]{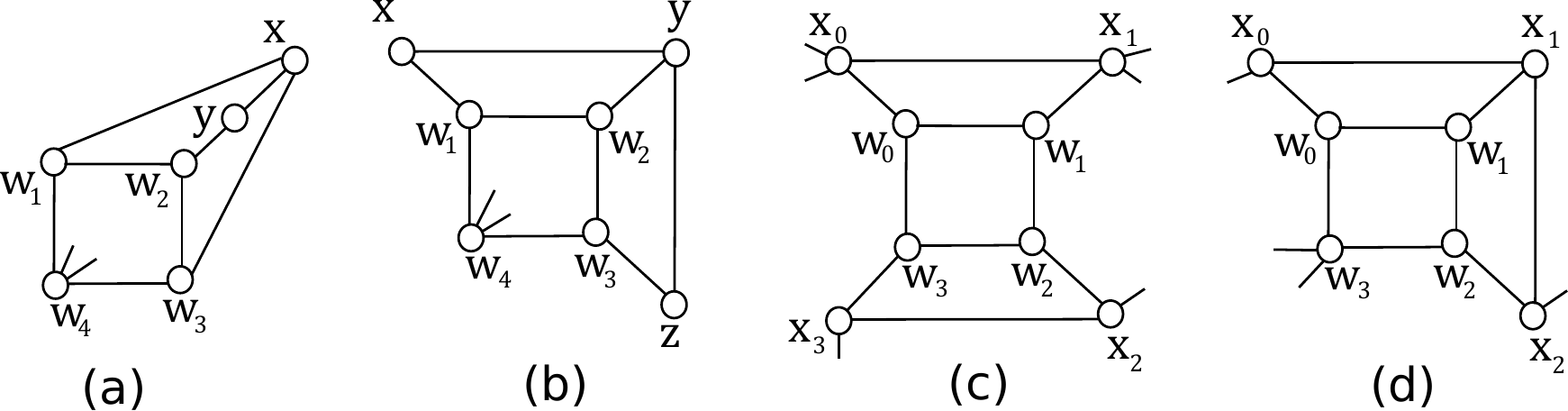}
      \caption{(a) A configuration in the proof of Claim~\ref{clm:sepC-thee-3-vers} when $x= z$ (b) A configuration in the proof of Claim~\ref{clm:sepC-thee-3-vers} when $x \not= z$ (c) A configuration in the proof of Lemma~\ref{lm:no-4face-4-4vers} (d) A configuration in the proof of Lemma~\ref{lm:no-4face-3-4vers}}
  \label{fig:sep-4cycle}
\end{figure} 

\begin{claim} \label{clm:sepC-atmost-one-deg-3}
Any separating $4$-cycle of $G$ must have at most one 3-vertex. 
\end{claim}
\begin{proof}
Let $w_1w_2w_3w_4$ be a separating $4$-cycle, denoted by $C$, of $G$ that has at least two 3-vertices. By Claim~\ref{clm:sepC-thee-3-vers}, $C$ has exactly two $3$-vertices.  By (i) of Claim~\ref{clm:C-non-adj-deg-3}, we assume that $w_1,w_2$ are two $3$-vertices of $C$ and their non-$C$ edges are embedded outside $C$. Since $C$ is separating and $G$ is two-edge connected, at least two non-$C$ edges of $C$ must be embedded inside $C$. By (ii) of Claim~\ref{clm:C-non-adj-deg-3}, two non-$C$ edges of any 4-vertex of $C$ must be embedded in the same side of $C$. We assume~w.l.o.g that two non-$C$ edges of $w_3$ are embedded inside $C$. Let $u$ and $v$ be non-$C$ neighbors of $w_1$ and $w_2$, respectively. By Claim~\ref{clm:3-4-adj}, $v$ must be a common neighbor of $u$ and $w_2$.  If $v$ is a 4-vertex, let $H$ be the subgraph of $G$ induced by $\{w_1,w_2,v,w_4\}$. Observe that we can collect $\{w_1,w_2\}$ from $H$. Since $v$ may be a neighbor of $w_4$,  $m(G\setminus H) \leq m(G) - 10$. By Observation~\ref{obs:meta-exclusion} with $L = H$ and $(\alpha,\beta,\gamma,\eta,\lambda) =(4,10,0,0,2)$, $2 - 4a + 10 b $ must be negative, contradicting Inequality~\eqref{lpc:2v-4v}. Thus, we can assume that $v$ is a 3-vertex. Let $K$ be a subgraph of $G$ induced by $\{u,v,w_1,w_2,w_4\}$. We can collect $\{v,w_2,w_1\}$ from $K$. Since $v$ may be a neighbor of $w_4$,  $m(G\setminus K) \leq m(G) - 10$. By Observation~\ref{obs:meta-exclusion} with $L = K$ and $(\alpha,\beta,\gamma,\eta,\lambda) =(5,10,0,0,3)$, $3 - 5a + 10b $ must be negative, contradicting Inequality~\eqref{lpc:min-deg-3a}.
\end{proof}

\subsection{Excluding a 4-face with at least one 3-vertex}\label{subsec:ex-4face-3vers}

\subsubsection{Excluding a 4-face with exactly four 3-vertices}

In this subsection, we denote $C = w_0w_1w_2w_3$ to be a 4-face of $G$ such that each $w_i$ is a 3-vertex, $0 \leq i \leq 3$. Let $X = \{x_0,x_1,x_2,x_3\}$ where each $x_i$ is the non-$C$ neighbor of $w_i$. All indices in this subsection are mod $4$ and to simplify the presentation, we write $w_j$ ($x_j$) instead of writing $w_{j \bmod 4}$ ($x_{j\bmod 4}$). 

\begin{claim}\label{clm:4-face-4-3-vers-sub}
Vertices in $X$ are pairwise distinct and $x_j$ is not adjacent to $x_{j+2}$ for any $j \in \{0,1\}$. 
\end{claim}
\begin{proof}
To prove that vertices in $X$ are pairwise distinct, we only need to prove that $x_j \not= x_{j+2}$ since $G$ is triangle-free. If $x_0 = x_2$, then $w_0w_1w_2x_0$ is a  separating $4$-cycle with at least three 3-vertices. If $x_1 = x_3$, then $w_0w_1x_1w_3$ is a  separating $4$-cycle with at least three 3-vertices. Both cases contradict  Claim~\ref{clm:sepC-atmost-one-deg-3}.

We now show that  $x_j$ and $x_{j+2}$ are non-adjacent. By symmetry, it suffices to show the non-adjacency of $x_0$ and $x_2$. Suppose otherwise. By planarity, $x_1$ and $x_3$ cannot be adjacent and if they have a common neighbor, it must be $x_0$ or $x_2$. Since $G$ is triangle-free, both $\{x_0,x_2\}$ cannot be common neighbors of $x_1$ and $x_3$. We assume~w.l.o.g that $x_2$ is a non-common neighbor of $x_1$ and $x_3$. We consider two cases:

\begin{description}
\item[Case 1] Vertex $x_0$ is a 3-vertex. Then, $x_1$ and $x_3$ has no common neighbor. Let $H$ be the subgraph induced by $\{w_0,w_1,w_2,w_3,x_2\}$.  We can collect $\{w_1,w_2,w_3\}$ from $H$. By Claim~\ref{clm:meta-cub-excl}, $p(G\setminus K) = q(G\setminus K) = 0$. Since $m(G\setminus K)\leq m(G)-10$, by Observation~\ref{obs:meta-exclusion} with $L = K$ and $(\alpha,\beta,\gamma,\eta,\lambda) =(5,10,0,0,3)$, $3 - 5a + 10b $ must be negative, contradicting Inequality~\eqref{lpc:min-deg-3a}.

\item[Case 2] Vertex $x_0$ is a 4-vertex.  Let $G'$ be the graph obtained by removing $\{x_0,w_0,w_1,w_2,w_3\}$ from $G$ and adding edge $x_1x_3$. $G'$ is triangle-free since common neighbors of $x_1$ and $x_3$ are all removed. By Claim~\ref{clm:meta-cub-excl}, $G\setminus C$ contains no $Q_3^{1-}$ subgraph and $T_6$ component. Thus, $p(G') + q(G') \leq 1$. Let $F'$ be the largest induced forest in $G'$. Observe that we can add $\{w_0,w_1,w_3\}$ to $F'$ to get an induced forest in $G$. Since  $G'$ has strictly smaller order than $G$, $F'$ has order at least $an(G') - bm(G') - cp(G') - dq(G')$.  Since $n(G') = n(G) - 5$ and $m(G') \leq m(G) - 10$,  by adding $\{w_0,w_1,w_3\}$ to $F'$, we get an induced forest in $G$ of order at least:
\begin{equation*}
\begin{split}
an(G') - bm(G') - cp(G') - dq(G') &\geq  an(G) - bm(G) + 3 - 5a + 10b - cp(G') - dq(G')
\end{split}
\end{equation*}
By Inequality~\eqref{lpc:T-6-4-5a} and Inequality~\eqref{lpc:T-6-4-5b}, $3 - 5a + 10b - c$ and $3 - 5a + 10b - d$ are both non-negative. Since $p(G')+ q(G') \leq 1$, $3 - 5a + 10b - cp(G') - dq(G')$ is non-negative. Thus, $G$ has an induced forest of order at least $an(G) - bm(G)$, contradicting that $G$ is a counter-example.  	 
\end{description}
\end{proof}

\begin{claim} \label{clm:two-cons-edges-one-4-face}
 At least one of two edges $w_jw_{j+1}$ and $w_{j+1}w_{j+2}$, for any $j$ in $\{0,1,2,3\}$, is not on the boundary of a $5^{+}$-face.
\end{claim}

\begin{proof}
Suppose that there exists $j \in \{0,1,2,3\}$ such that $w_jw_{j+1}$ and $w_{j+1}w_{j+2}$ are on the boundaries of $5^+$-faces. We assume~w.l.o.g that $j = 0$. Let $G'$ be the graph obtained from $G$ by removing $\{w_0,w_2,w_3\}$ and adding two edges $x_0w_1, w_1x_2$. We observe that, by construction, $x_3$ is the only possible 2-vertex of $G'$. Thus, $G'$ contains no $T_6$ component. We consider two cases:
\begin{description}
\item[Case 1]Graph $G'$ contains no $Q_3^{1-}$ subgraph. Then $G'$ has an induced forest $F'$ of order at least $an(G') - bm(G')$. Since $n(G') = n(G)-3$ and $m(G') = m(G)-5$, by adding $\{w_0,w_2\}$ to $F'$, wet get an induced forest of $G$ of order at least:
	\begin{equation*}
	a(n(G)-3) - b(m(G)-5) + 2 = an(G) -bm(G) + 2 -3a + 5b
	\end{equation*}
Since $ 2 -3a + 5b \geq 0$ by Inequality~\eqref{lpc:2v-2vb}, $G$ has an induced forest of order  at least $an(G) - bm(G)$, contradicting that $G$ is a counter-example. 
\item[Case 2]Graph $G'$ contains at least one $Q_3^{1-}$ subgraph. By Claim~\ref{clm:meta-cub-excl}, $G-\{w_0,w_2,w_3\}$ contains no $Q_3^{1-}$ subgraph. Thus, any $Q_3^{1-}$ subgraph of $G$ must contain $w_1$.  By Observation~\ref{ob:Q-3-dis}, $G'$ has exactly one $Q_3^{1-}$ subgraph.  If $G'$ contains $Q_3$, then the subgraph of $G$ induced by $V(Q_3)\cup \{w_0,w_2,w_3\}$ has degree 1 in $G$, contradicting that $G$ is two-edge connected. Thus, we can assume that $G'$ contains a $Q_3^1$ subgraph $K$. Let $G'' = G'\setminus K$. We observe that $G''$ can also be obtained from $G$ by removing $V(K)\cup \{w_0,w_2,w_3\}$. Since $V(K) \cup \{w_0,w_{2},w_{3}\}$  induces a connected subgraph of $G$, $p(G'') = q(G'') = 0$ by Claim~\ref{clm:meta-cub-excl}. Thus, $G''$ has a forest $F''$ of order at least $an(G'') -bm(G'')$. By Lemma~\ref{lm:cublike-forest}, we can collect $5$ vertices from $K$ to obtain an induced forest $F'$ of $G'$ of order at least  $an(G'') -bm(G'') + 5$. By adding $\{w_0,w_2\}$ to $F'$, we get an induced forest $F$ of $G$ of order at least $an(G'') -bm(G'') + 7$. Since $n(G'') = n(G)-11$ and $m(G'') = m(G)-18$, $F$ has order at least:
\begin{equation*}
an(G) -bm(G) + 7 - 11a + 18b
\end{equation*}
We obtain contradiction by adding Inequality~\eqref{lpc:2e-one-4f} to $\lp$.
\begin{equation} \label{lpc:2e-one-4f}
7 - 11a + 18b  \geq 0
\end{equation}
\end{description} 
\end{proof}

\begin{claim}\label{clm:2-4-vers-in-U}
At least one of two vertices $x_j,x_{j+2}$ is a 3-vertex, for any $j$ in $\{0,1\}$.
\end{claim}
\begin{proof}

Suppose that $x_j$ and $x_{j+2}$ are two 4-vertices for some $j \in \{0,1\}$. Let $H$ be the graph induced by $V(C) \cup \{x_j,x_{j+2}\}$. Observe that we can collect $\{w_j,w_{j+1}, w_{j+2}\}$ from $H$. Since $m(G\setminus H) \leq m(G)- 14$, by Observation~\ref{obs:meta-exclusion} with $L = H$ and $(\alpha,\beta,\gamma,\eta,\lambda) =(6,14,0,0,3)$, $3 - 6a + 14b $ must be negative.  Thus, we obtain contradiction by adding Inequality~\eqref{lpc:2v-4vX} to $\lp$.
\begin{equation} \label{lpc:2v-4vX}
3 - 6a + 14b \geq 0
\end{equation}
\end{proof}

\begin{lemma}\label{lm:no-4face-4-4vers}
Graph $G$ has no 4-face with four 3-vertices.
\end{lemma}
\begin{proof}
Let $H$ be the subgraph of $G$ induced by $V(C)\cup X$. If no edge of $C$ is on the boundary of a $5^{+}$-face, then $H$ is a $Q_3^{4-}$ subgraph of $G$, contradicting Claim~\ref{clm:exclue-Q-3-4}. Thus, we can assume at least one edge of $C$ is on the boundary of a $5^{+}$-face. By Claim~\ref{clm:two-cons-edges-one-4-face}, $C$ has at most two edges on the boundaries of $5^+$-faces and they cannot be incident to the same vertex of $C$. Thus, there exists $j \in \{0,1\}$ such that two edges $w_jw_{j+1}$ and $w_{j+2}w_{j+3}$ are only on the boundary of 4-faces. Without loss of generality, we assume that $j = 0$. Thus, $x_0x_1$ and $x_2x_3$ are edges of $G$ (see Figure~\ref{fig:sep-4cycle}(c)). By symmetry, we can assume that $x_0$ is the highest degree vertex of $X$. By Claim~\ref{clm:2-4-vers-in-U}, $x_2$ is a 3-vertex. 

We claim that (i) $x_1$ is a 4-vertex and (ii) $x_1x_2$ and $x_3x_0$ are non-edges of $G$. Suppose that at least one of two claims fails, we show that we can collect 5 vertices from $H$. If $x_1x_2$ is an edge of $G$, then we can collect $\{x_1,w_1,w_0,w_3,x_2\}$ from $H$. If $x_0x_3$ is an edge of $G$, then we can collect $\{x_0,w_0,w_1,w_3,x_2\}$ from $H$. If $x_1$ is a 3-vertex, then we can collect $\{w_0,w_2,w_3,x_2,x_1\}$ from $H$. Since $m(G\setminus H) \leq m(G)-13$, by Observation~\ref{obs:meta-exclusion} with $L = H$ and $(\alpha,\beta,\gamma,\eta,\lambda) =(8,13,0,0,5)$, $5 - 8a + 13b $ must be negative, contradicting Inequality~\eqref{lpc:Q-3-1a}. Thus, both claims hold.

Let $K$ be the subgraph induced by $V(C) \cup \{x_0,x_2,x_3\}$. By Claim~\ref{clm:2-4-vers-in-U}, $x_3$ is a 3-vertex. Thus, we can collect $\{x_3,w_0,w_2,w_3\}$ from $K$. Since $x_0$ is the highest degree vertex of $C$, $x_0$ is a 4-vertex. Thus, $m(G\setminus K) \leq m(G)-14$. By Observation~\ref{obs:meta-exclusion} with $L = K$ and $(\alpha,\beta,\gamma,\eta,\lambda) =(7,14,0,0,4)$, $4 -7a + 14b $ must be negative, contradicting Inequality~\eqref{lpc:sep4-three-3v}.
\end{proof}

\noindent By combining Lemma~\ref{lm:no-4face-4-4vers} and Claim~\ref{clm:sepC-atmost-one-deg-3}, we get:
\begin{corollary}\label{cor:no-4Cycle-4-3vers}
Graph $G$ has no 4-cycle with four 3-vertices.
\end{corollary}
\subsubsection{Excluding a 4-face with at least two 3-vertices}

\begin{lemma}\label{lm:no-4face-3-4vers}
Graph $G$ has no 4-face with three 3-vertices.
\end{lemma}
\begin{proof}
Let $C = w_0w_1w_2w_3$ be a 4-face of $G$ that has three 3-vertices, say $w_0,w_1,w_2$. Suppose that $w_i$ and $w_{i+2}$ share a neighbor, say $x$, for some $i$ in $\{0,1\}$. Then, the cycle $xw_iw_{i+1}w_{i+2}$ is a separating 4-cycle that has at least two 3-verties, contradicting Claim~\ref{clm:sepC-atmost-one-deg-3}. Thus, $w_i$ and $w_{i+2}$ have no common neighbor for any $i$ in $\{0,1\}$. Let $x_0,x_1,x_2$ be the neighbors of $w_0,w_1,w_2$, respectively. By Claim~\ref{clm:3-4-adj}, $x_0x_1$ and $x_1x_2$ are edges of $G$ (see Figure~\ref{fig:sep-4cycle}(d)). Let $H$ be the subgraph of $G$ induced by $V(C) \cup \{x_0,x_1,x_2\}$. By Corollary~\ref{cor:no-4Cycle-4-3vers}, at least one vertex in $\{x_0,x_1,x_2\}$ is a 4-vertex. Thus, $m(G\setminus H) \leq m(G) - 14$. Observe that we can collect $\{x_1,w_0,w_1,w_2\}$ from $H$.  By Observation~\ref{obs:meta-exclusion} with $L = H$ and $(\alpha,\beta,\gamma,\eta,\lambda) =(7,14,0,0,4)$, $4 -7a + 14b $ must be negative, contradicting Inequality~\eqref{lpc:sep4-three-3v}.
\end{proof}

\noindent By Lemma~\ref{lm:no-4face-3-4vers} and Claim~\ref{clm:sepC-atmost-one-deg-3}, we have:
\begin{corollary}\label{cor:no-4Cycle-3-3vers}
Graph $G$ has no 4-cycle with at least three 3-vertices.
\end{corollary}

\begin{lemma}\label{lm:no-4face-2-4vers}
Graph $G$ has no 4-face with exactly two 3-vertices.
\end{lemma}
\begin{proof}
Let $w_0,w_1,w_2,w_3$ be vertices in clock-wise order of a 4-face $C$ of $G$ that has exactly two 3-vertices. By Claim~\ref{clm:C-non-adj-deg-3} two 3-vertices of $C$ must be adjacent. Without loss of generality, we assume that two 3-vertices are $w_0$ and $w_1$. Let $x_0,x_1$ be the neighbors of $w_0,w_1$, respectively. By Claim~\ref{clm:3-4-adj}, $x_0x_1$ is an edge of $G$. By Corollary~\ref{cor:no-4Cycle-3-3vers}, $x_0$ and $x_1$ are 4-vertices. We now show that $x_j$ and $w_{j+2}$ are non-adjacent for any $j$ in $\{0,1\}$. If $x_0$ and $w_2$ are adjacent, then the cycle $x_0w_0w_{1}w_{2}$ is a separating 4-cycle that has at least two 3-vertices. If $x_1$ and $w_3$ are adjacent, then the cycle $x_1w_1w_{0}w_{3}$ is a separating 4-cycle that has at least two 3-vertices. Both cases contradict Claim~\ref{clm:sepC-atmost-one-deg-3}. Thus, $x_j$ and $x_{j+2}$ are non-adjacent.

If $x_0$ and $w_2$ share a common neighbor and $x_1$ and $w_3$ share a common neighbor, by planarity, they all share a common neighbor, contradicting that $G$ is triangle-free. Thus, by symmetry (see Figure~\ref{fig:fourF}(a)), we can assume that $x_0$ and $w_2$ share no common neighbor. Let $G'$ be the graph obtained by removing $\{x_1,w_0,w_1,w_3\}$ from $G$ and adding edge $x_0w_2$. Then, $G'$ is a triangle-free planar graph. By Claim~\ref{clm:meta-cub-excl}, the graph obtained by removing $\{x_1,w_0,w_1,w_3\}$ from $G$ has no $T_6$ component and $Q_3^{1-}$ subgraph. Thus, any $T_6$ component or $Q_3^{1-}$ subgraph of $G'$ must contains edge $x_0w_2$. That implies $p(G') + q(G') = 1$.  We consider three cases:
\begin{description}
\item[Case 1] $p(G') = q(G') = 0$. Then $G'$ has an induced forest $F'$ of order at least $an(G') - bm(G')$. By adding $\{w_0,w_1\}$ to $F'$, we obtain a forest of order at least  $an(G') - bm(G') + 2$. Since $n(G')=  n(G) - 4$ and $m(G') = m(G)-10$, we have:
\begin{equation*}
an(G') - bm(G') + 2 = an(G) -bm(G) + 2 - 4a + 10b
\end{equation*}
Since $2 - 4a + 10b$ is non-negative by Inequality~\eqref{lpc:2v-4v}, $G$ has an induced forest of order at least $an(G) - bm(G)$, contradicting that $G$ is a counter-example.
\item[Case 2] $p(G') = 1$ and $q(G') = 0$. Let $G''$ be the graph obtained from $G'$ by removing the $T_6$ component of $G'$. Since $G''$ can also be obtained from $G$ by removing $V(T_6)\cup\{x_1,w_0,w_1,w_3\}$ which induces a connected subgraph of $G$, by Claim~\ref{clm:meta-cub-excl}, $p(G'') = q(G'') = 0$. Thus, $G''$ has a forest $F''$ of order at least $an(G'') - bm(G'')$. By Lemma~\ref{lm:T6-like-forest}, we can add 4 vertices from the $T_6$ component to $F''$ to get an induced forest $\hat{F}$ of $G'$ of order at least $an(G'') - bm(G'')+4$. By adding $w_0$ and $w_1$ to $\hat{F}$, we get an induced forest of order at least $an(G'') - bm(G'')+6$ in $G$. Since $n(G'') = n(G)-10$ and $m(G'') = m(G)-18$, we have:
\begin{equation*}
an(G'') - bm(G'') + 6 = an(G) -bm(G) + 6 - 10a + 18b
\end{equation*}
Since $6 - 10a + 18b$ is non-negative by Inequality~\eqref{lpc:min-deg-3a}, $G$ has an induced forest of order at least $an(G) -bm(G)$, contradicting that $G$ is a counter-example.

\item[Case 3] $p(G') = 0$ and $q(G') = 1$. Let $M$ be the $Q_3^{1-}$ subgraph of $G'$. We consider two subcases:
\begin{description}
\item[Subcase 1] $M$ is $Q_3^{1}$ in $G'$. Let $G''' = G'\setminus M$. Then, $G'''$ can also be obtained from $G$ by removing $V(M) \cup \{x_1,w_0,w_1,w_3\}$ which induces a connected subgraph of $G$. Thus, by Claim~\ref{clm:meta-cub-excl}, $p(G''') = q(G''') = 0$. Let $F'''$ be a forest of $G'''$ of order at least $an(G''') - bm(G''')$. By Lemma~\ref{lm:cublike-forest}, we can add 5 vertices of $M$ to $F'''$ to get an induced forest $\overline{F}$ in $G'$ of order at least $an(G''') - bm(G''')+5$. By adding $w_0$ and $w_1$ to $\overline{F}$, we get an induced forest of order at least $an(G''') - bm(G''')+7$ in $G$. Since $n(G''') = n(G)-12$ and $m(G'') = m(G)-23$, we have:
\begin{equation*}
an(G''') - bm(G''') + 6 = an(G) -bm(G) + 7 - 12a + 23b
\end{equation*}
 Thus, we obtain contradiction by adding Inequality~\eqref{lpc:4f-2-3v} to $\lp$.
\begin{equation} \label{lpc:4f-2-3v}
7 - 12a + 23b \geq 0
\end{equation}
\item[Subcase 2]$M$ is $Q_3$ in $G'$. Recall that $x_0w_2$ must be an edge of $M$. By symmetry of $Q_3$, we can assume w.l.o.g that $x_0 = u_1$ and $w_2 = u_2$ (see Figure~\ref{fig:fourF}(b)). Consider the cycle $\hat{C} = x_0u_5u_6w_2u_3u_4$ of $M$. $\hat{C}$ is also the cycle of $G$. Thus, $x_1,w_0,w_1,w_3$ is embedded inside $\hat{C}$ in $G$.  That implies $u_7$ and $u_8$ are 3-vertices in $G$. Observe that the path $x_0w_0w_1w_2$ separate the internal part of  $\hat{C}$ into two parts, one contains $x_1$ and another contains $w_3$. Let $C' = x_0w_0w_1w_2u_6u_5$ and $C'' = x_0w_0w_1w_2u_3u_4$ be two cycles of $G$. We consider $C$ as a 4-cycle of $G$ instead of a 4-face so that we can speak of the symmetry of the subgraph induced by $\{x_0,x_1,w_0,w_1,w_2,w_3\}$ . By symmetry, we can assume~w.l.o.g that $x_1$ is inside $C'$ and $w_3$ is inside $C''$. Since $G$ is triangle free, $x_1$ and $u_5$ are non-adjacent. Thus, $u_5$ is a 3-vertex in $G$. That implies 4-cycle $u_5u_6u_7u_8$ has three 3-vertices, contradicting Corollary~\ref{cor:no-4Cycle-3-3vers}.
\end{description}
\end{description}
\end{proof}
\begin{figure}[tbh]
  \centering
   \includegraphics[height=1.4in]{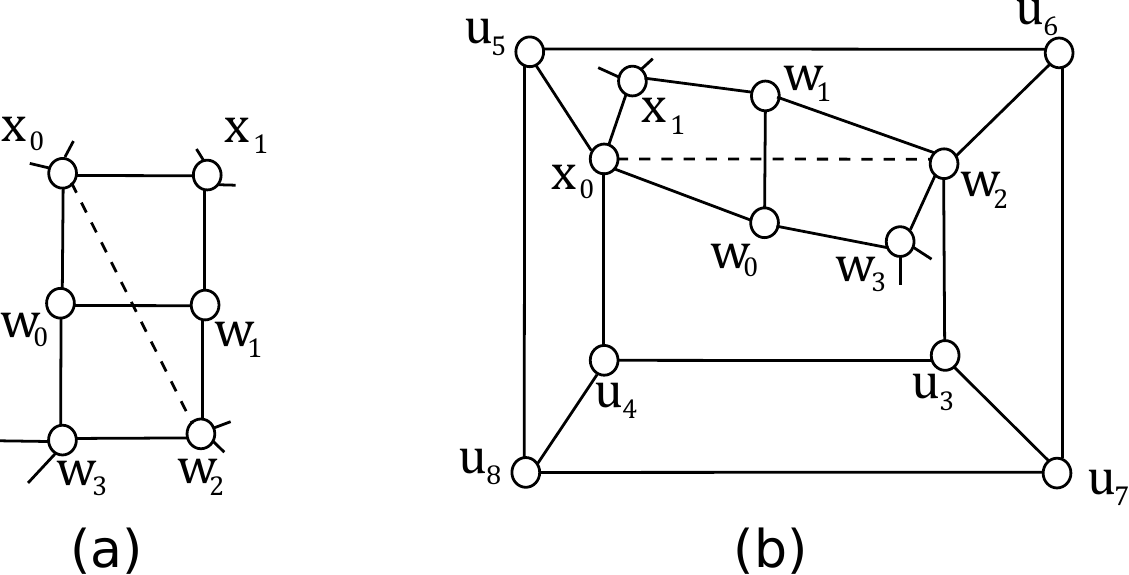}
      \caption{(a) A configuration in the proof of Lemma~\ref{lm:no-4face-2-4vers} (b) A configuration in the proof of Subcase 2 in Lemma~\ref{lm:no-4face-2-4vers}.}
  \label{fig:fourF}
\end{figure} 

\subsubsection{Excluding a 4-face with exactly one 3-vertices}

In this subsection, we denote $C = w_0w_1w_2w_3$ to be a 4-face of $G$ such that $w_0$ is a 3-vertex and $w_1,w_2,w_3$ are 4-vertices. From Lemma~\ref{lm:no-4face-4-4vers},~\ref{lm:no-4face-3-4vers} and ~\ref{lm:no-4face-2-4vers} and Claim~\ref{clm:sepC-atmost-one-deg-3}, we have:
\begin{corollary}\label{cor:4face-one-4ver}
Any $4$-cycle of $G$ has at most one 3-vertex.
\end{corollary}

\begin{claim}\label{clm:no-3-3-4}
Graph $G$ has no 3-vertex that has a 3-vertex and a 4-vertex as neighbors.
\end{claim}
\begin{proof}
Suppose that $G$ has a 3-vertex $u$ that has a 3-vertex $v$ and a 4-vertex $w$ as neighbors. Let $x$ be a neighbor of $u$ such that $x \not\in \{v,w\}$. By Claim~\ref{clm:3-4-adj}, $x$ and $v$ has a neighbor $y$ such that $y\not= u$. Thus, 4-cycle $uxyv$ has two 3-vertices, contradicting Corollary~\ref{cor:4face-one-4ver}.
\end{proof}

Let $x_0$ be the non-$C$ neighbor of $w_0$. By Claim~\ref{clm:3-4-adj}, $x_0$ and $w_1$ have a common neighbor, say $x_1$, and $x_0$ and $w_3$ have a common neighbor, say $x_3$.  Since $x_0x_1w_1w_0$ is a 4-cycle that has $w_0$ as a 3-vertex, by Corollary~\ref{cor:4face-one-4ver}, $x_0$ must be a 4-vertex. 
\begin{claim}\label{clm:v0-v2-no-comm-nbr}
Two vertices $x_0,w_2$ are non-adjacent.
\end{claim}
\begin{proof}
Suppose otherwise. Let $x_2$ be the non-$C$ neighbor of $w_2$ such that $x_2 \not= x_0$. Let $C_1 = x_0w_0w_3w_2$ and $C_2 = x_0w_0w_1w_2$ be two 4-cycles of $G$. By (ii) of Claim~\ref{clm:C-non-adj-deg-3}, two edges $w_2x_2$ and $w_2w_1$ must be embedded in the same side of $C_1$. That  implies two edges $w_2x_2$ and $w_2w_3$ are embedded in different sides of $C_2$, contradicting (ii) of Claim~\ref{clm:C-non-adj-deg-3}.
\end{proof}

\begin{claim}\label{clm:v1-v3-no-comm-nbr}
There is no common neighbor between $w_1$  and $w_3$.
\end{claim}
\begin{proof}
We note that $w_1$ and $w_3$ can have up to 4 common neighbors. Let $x$ be a non-$C$ common neighbor of  $w_1$ and $w_3$. Since $w_0$ is a 3-vertex in 4-cycle $w_1w_0w_3x$, by Corollary~\ref{cor:4face-one-4ver}, $x$ must be a 4-vertex. We consider two cases:
\begin{description}
\item[Case 1] Three vertices $x_0,w_1,w_3$ share a common neighbor, that we assume~w.l.o.g to be $x$ (see Figure~\ref{fig:fourF-one3v}(a)). Let $C_3 = x_0w_0w_1x$, $C_4 = x_0w_0w_3x$ and $C_5 = xw_3w_0w_1$. By (ii) of Claim~\ref{clm:C-non-adj-deg-3}, two non-$C_3$ edges incident to $x$ must be embedded in the same side of $C_3$ and two non-$C_4$ edges incident to $x$ must be embedded in the same side of $C_4$. That implies two non-$C_5$ edges incident to $x$ are embedded in different side of $C_5$, contradicting (ii) of Claim~\ref{clm:C-non-adj-deg-3}.
\item[Case 2] Three vertices $x_0,w_1,w_3$ do not share a common neighbor. Then, $x, x_1$ and $x_3$ are pair-wise distinct(see Figure~\ref{fig:fourF-one3v}(b)). Let $H$ be the subgraph of $G$ induced by $V(C) \cup \{x,x_0,x_1,x_3\}$. Observe that we can collect $\{x_0,w_0,w_1,w_3\}$ from $H$. Since $G$ is triangle-free, $xw_2, xx_3,xx_1,x_1x_3$ are non-edges of $G$. Thus, $n(H) = n(G)-8$ and $m(H) = m(G) - 20$. By Observation~\ref{obs:meta-exclusion} with $L = H$ and $(\alpha,\beta,\gamma,\eta,\lambda) =(8,20,0,0,4)$, $4 -8a + 20b $ must be negative, contradicting Inequality~\eqref{lpc:2v-4v}.\qedhere
\end{description}
\end{proof}

\begin{figure}[tbh]
  \centering
   \includegraphics[height=1.2in]{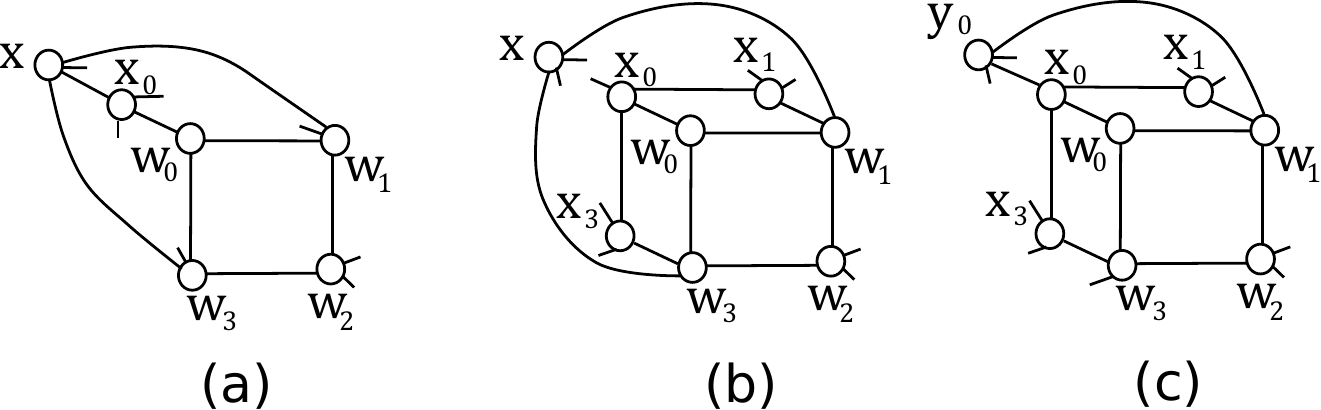}
      \caption{(a) A configuration in the proof of Case 1 of Claim~\ref{clm:v1-v3-no-comm-nbr} (b) A configuration in the proof of Case 2 Claim~\ref{clm:v1-v3-no-comm-nbr} (c) A configuration in the proof of Claim~\ref{clm:xyz-adj}}
  \label{fig:fourF-one3v}
\end{figure} 

Let $y_1$ and $y_3$ be the non-$C$ neighbors of $w_1$ and $w_3$, respectively, such that $y_1\not = x_1$ and $y_3\not= x_3$. Let $y_0$ be the non-$C$ neighbor of $x_0$ such that $y_0 \not\in \{x_1,x_3\}$. Let $Z = V(C) \cup \{x_0,x_1,x_3, y_0,y_1,y_3\}$.
\begin{claim}\label{clm:Z-distinct}
Vertices in $Z$ are pairwise distinct.
\end{claim}
\begin{proof}
 By Claim~\ref{clm:v1-v3-no-comm-nbr}, two vertices $x_1$ and $x_3$ are distinct and two vertices $y_1$ and $y_3$ are distinct. Since $G$ is triangle-free, $y_0\not= x_3$ and $y_0 \not= x_1$. To prove the claim, we only need to prove that $y_0 \not= y_1$ and $y_0\not= y_3$. By symmetry, it suffices to prove $y_0 \not= y_1$. Suppose otherwise. Let $H$ be the subgraph of $G$ induced by $V(C) \cup \{x_0,x_1,x_3,y_0\}$ (see Figure~\ref{fig:fourF-one3v}(c)). Since $G$ is triangle-free, $y_0$ and $x_3$ are non-adjacent and $y_0$ and $w_2$ are non-adjacent. By Claim~\ref{clm:v1-v3-no-comm-nbr}, $y_0$ and $w_3$ are non-adjacent. Thus, $m(G\setminus H) = m(G)-20$. Since we can collect $\{x_0,w_0,w_1,w_3\}$ from $H$, by Observation~\ref{obs:meta-exclusion} with $L = H$ and $(\alpha,\beta,\gamma,\eta,\lambda) =(8,20,0,0,4)$, $4 -8a + 20b $ must be negative, contradicting Inequality~\eqref{lpc:2v-4v}. \qedhere
\end{proof}

\begin{claim}\label{clm:xyz-adj}
At least one of $y_0y_1, y_1y_3,y_0y_3$ is an edge of $G$.
\end{claim}
\begin{proof}
Suppose that $y_0y_1, y_1y_3,y_0y_3$ are non-edges of $G$. Let $N = \{x_0,x_1,x_3,w_1,w_2,w_3\}$. Let $G'$ be the graph obtained from $G$ by removing vertices in $N$ and adding edges $\{w_0y_0,w_0y_1,w_0y_3\}$ (see Figure~\ref{fig:fourF-symmetry}(a)). By Claim~\ref{clm:meta-cub-excl}, the graph obtained from $G$ by removing vertices in $N$ has no $T_6$ component and $Q_3^{1-}$ subgraph. Thus, any $T_6$ component and $Q_3^{1-}$ subgraph of $G'$ must contain $w_0$. That implies $p(G') + q(G') \leq 1$. If $q(G') = 1$, let $H$ be a $T_6$ component of $G$. Then $w_0$ must be a 3-vertex of $H$. Since any 3-vertex of a $T_6$ component is adjacent to a 2-vertex, at least one neighbor of $w_0$ must be a 2-vertex in $G'$. However, $w_0$'s neighbors all are $3^+$-vertices in $G'$. Thus, $q(G') = 0$. Since $G$ is a counter-example of minimal order, $G'$ has an induced forest $F'$ of order at least $an(G') - b m(G') - cp(G')$. By adding $x_0,w_1,w_3$ to $F'$, we obtain an induced forest $F$ of order at least $an(G') - b m(G') -cp(G')+ 3$ in $G$. Since $n(G') = n(G)-6$ and $m(G') = m(G) - 15$, we have:
\begin{equation*}
an(G') - b m(G')  - cp(G')+ 3\geq an(G) - bm(G) + 3 - 6a + 15b - c
\end{equation*}
 Thus, we obtain contradiction by adding Inequality~\eqref{lpc:xyz-adj} to $\lp$.
\begin{equation} \label{lpc:xyz-adj}
3 - 6a + 15b - c \geq 0
\end{equation}
\end{proof}
\begin{figure}[tbh]
  \centering
   \includegraphics[height=1.2in]{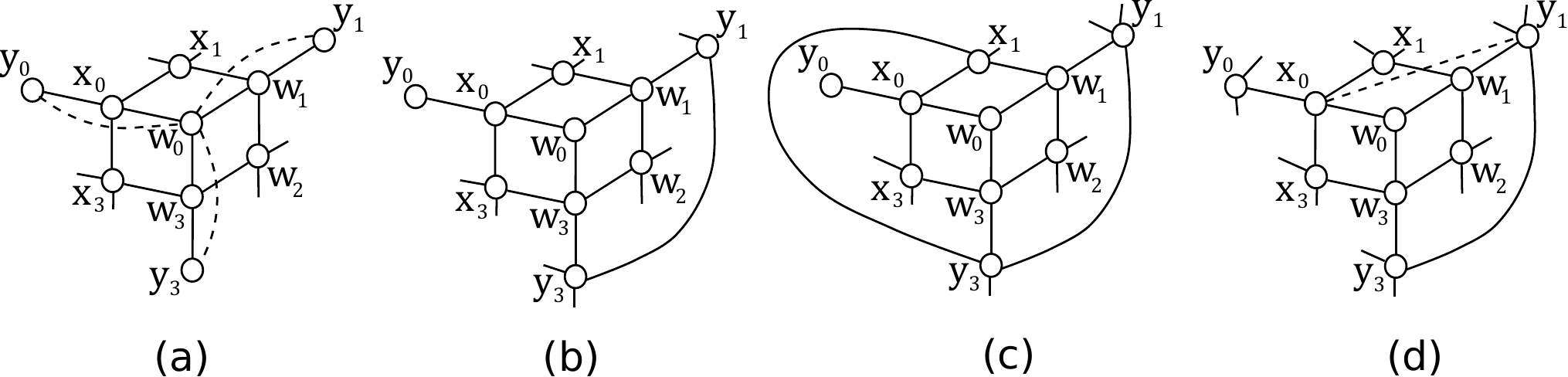}
      \caption{(a) A configuration in the proof of Claim~\ref{clm:xyz-adj} (b) A configuration in the proof of Claim~\ref{clm:x-y-4vers} (c) A configuration in the proof of Claim~\ref{clm:y-u1-non-adj} (d)  A configuration in the proof of Lemma~\ref{lm:no-4face-1-3vers}}
  \label{fig:fourF-symmetry}
\end{figure} 
\noindent Herein, we regard 4-face $C$ as a 4-cycle so that we can speak of the symmetry of neighbors of $w_0$ in $G$ (see Figure~\ref{fig:fourF-symmetry}(a)).  By symmetry, we can assume w.l.o.g that $y_1y_3$ is an edge of $G$.
\begin{claim}\label{clm:x-y-4vers}
Two vertices $y_1$ and $y_3$ are 4-vertices. 
\end{claim}
\begin{proof}
Suppose otherwise. We can assume~w.l.o.g that  $y_1$ is a 3-vertex. By Claim~\ref{clm:no-3-3-4}, $y_3$ must be a 4-vertex (see Figure~\ref{fig:fourF-symmetry}(b)). Let $H$ be the subgraph of $G$ induced by $\{w_0,w_1,w_3,x_0,x_1,x_3,y_1,y_3\}$. Since $G$ is triangle-free, two vertices $y_1$ and $x_1$ are non-adjacent and two vertices $x_3$ and $y_3$ are non-adjacent. By planarity, if $y_1x_3$ is an edge of $G$, then $x_1y_3$ is non-edge of $G$ and vice versa. Thus, $m(G\setminus H)\leq m(G) - 19$. Since we can collect $\{y_1,x_0,w_0,w_3\}$ from $H$, by Observation~\ref{obs:meta-exclusion} with $L = H$ and $(\alpha,\beta,\gamma,\eta,\lambda) =(8,19,0,0,4)$, $4 -8a + 19b $ must be negative.  Thus, we obtain contradiction by adding Inequality~\eqref{lpc:x-y-4v} to $\lp$.
\begin{equation} \label{lpc:x-y-4v}
4 -8a + 19b \geq 0
\end{equation}
\end{proof}

\begin{claim}\label{clm:y-u1-non-adj}
Two vertices $y_3,x_1$ are non-adjacent and two vertices $x_3,y_1$ are non-adjacent.
\end{claim}
\begin{proof}
By symmetry, we only need to prove $y_3$ and $x_1$ are non-adjacent. Suppose that $y_3x_1$ is an edge of $G$(see Figure~\ref{fig:fourF-symmetry}(c)). Let $C_6$ be cycle $ x_1w_1w_2w_3y_3$. Two vertices $y_1$ and $x_3$ are embedded in different sides of $C_6$. Let $H$ be the subgraph induced by $\{w_0,w_1,w_2,w_3,x_0,x_1,x_3,y_3\}$. Observe that we can collect $x_0,w_0,w_1,w_3$ from $H$. Since $m(G\setminus H) \leq m(G)-20$,  by Observation~\ref{obs:meta-exclusion} with $L = H$ and $(\alpha,\beta,\gamma,\eta,\lambda) =(8,20,0,0,4)$, $4 - 8a + 20b $ must be negative, contradicting Inequality~\eqref{lpc:2v-4v}.
\end{proof}

\begin{lemma}\label{lm:no-4face-1-3vers}
 Graph $G$ has no 4-face with exactly one 3-vertex.
\end{lemma}
\begin{proof}
Since $G$ is triangle-free, at most one of $y_0y_1,y_0y_3$ is an edge of $G$. By symmetry, we can assume w.l.o.g that $y_0$ and $y_1$ are non-adjacent. Let $J = \{w_0,w_1,w_2,w_3,x_1,x_3,y_3\}$. Let $G'$ be the graph obtained from $G$ by removing vertices in $J$ and adding edge $x_0y_1$  (see Figure~\ref{fig:fourF-symmetry}(d)). Since two vertices $y_0, y_1$ are non-adjacent and two vertices $y_1,x_3$ are non-adjacent by Claim~\ref{clm:y-u1-non-adj}, $G'$ is triangle-free. By Claim~\ref{clm:meta-cub-excl}, the graph obtained by removing $J$ from $G$ has no $T_6$ component and $Q_3^{1-}$ subgraph. Thus, any $T_6$ and $Q_3^{1-}$ subgraph of $G'$ must contains edge $x_0y_1$. Since $x_0$ is a $2$-vertex of $G'$, $G'$ has no $Q_3^{1-}$ subgraph. We now argue that $G'$ contains no $T_6$ component.

Suppose that $G'$ contains a $T_6$ component. Then, $x_0$ must be one of two 2-vertices of $T_6$. By symmetry of $T_6$, we can assume w.l.o.g that $x_0$ is $v_5$. Thus, edge $x_0y_1$ is $v_2v_5$ or $v_5v_3$.  Let $C_7$ be cycle $v_1v_2v_4v_3$. $C_7$ separates $v_6$ and $v_5$ in both $G'$ and $G$. Thus, $C_7$ also separates $v_6$ from every vertex reachable from $v_5$ in $G\setminus C_7$. That implies $v_6$ is also a 2-vertex in $G$, contradicting Lemma~\ref{lm:min-deg-3}. Thus, $G$ contains no $T_6$ component.

Since $G$ is a minimal counter-example, $G'$ has an induced forest $F'$ of order at least $an(G') - b m(G')$. Let $F = F'\cup\{w_0,w_1,w_3\}$. $F$ is an induced forest of $G$ of order at least $an(G') -bm(G') + 3$. Since $n(G') = n(G) - 7$ and $m(G') = m(G)-19$, we have:
\begin{equation*}
|F| \geq an(G) -bm(G) + 3 - 7 a + 19b
\end{equation*}
Thus, we obtain contradiction by adding Inequality~\eqref{lpc:4f-1-3v} to $\lp$.
\begin{equation} \label{lpc:4f-1-3v}
3-7a + 19b  \geq 0
\end{equation}
\end{proof}

\subsection{Excluding a 5-face with at least four 3-vertices}\label{subsec:ex-5face-4-3vers}

Let $w_0,w_1,w_2,w_3,w_4$  be vertices in clock-wise order of a 5-face $C$ of $G$ such that $C$ has at most one $4$-vertex. Let $X = \{x_0,x_1,x_2,x_3,x_4\}$ be a set of vertices such that $x_i$ is a non-$C$ neighbor of $w_i$ for all $0 \leq  i \leq 4$.
\begin{claim}\label{clm:5face-no-4ver}
Face $C$ has no 4-vertex.
\end{claim}
\begin{proof}
Suppose that $w_i$ is a $4$-vertex in $G$. Recall that $C$ has at least four $3$-vertices. Thus, $w_{i+1}$ is a 3-vertex that has a 3-vertex and a 4-vertex as neighbors, contradicting Claim~\ref{clm:no-3-3-4}
\end{proof}

\begin{observation}\label{obs:all-vers-U-dist}
Vertices in $X$ are $3$-vertices and pairwise distinct.
\end{observation}
\begin{proof}
 Suppose that $x_i = x_{i+2}$ for some $i \in \{0,1,2,3,4\}$ (indices are mod 5). Then, $w_iw_{i+1}w_{i+2}x_i$  is a 4-cycle that has at least three 3-vertices, contradicting Corollary~\ref{cor:4face-one-4ver}. The fact that verices in $X$ are 3-vertices follows directly from Claim~\ref{clm:no-3-3-4}.
\end{proof}

\begin{lemma}\label{lm:no-5face-many-3vers}
Any 5-face of $G$ has at least two 4-vertices.
\end{lemma}
\begin{proof}
 By Claim~\ref{clm:5face-no-4ver}, a 5-face $C$ that has at most one 4-vertex actually has no 4-vertex. By Corollary~\ref{cor:4face-one-4ver}, $x_i$ and $x_{i+1}$ are non-adjacent, for any $i$ such that $ 0 \leq i \leq 4$. Let $G'$ be the graph obtained from $G$ by removing $\{w_0,w_3,w_4\}$ and adding edges $x_0w_1,x_3w_2$.  $G'$ has no $T_6$ component since only $u_4$ is a 2-vertex in $G'$. Suppose that $G'$ contains a $Q_3^{1-}$ subgraph, say $H$, of $G$. Since the graph obtained from $G$ by removing $\{w_0,w_3,w_4\}$ has no $Q_3^{1-}$ subgraph, $H$ must contain at least one of two new edges $x_0w_1,x_3w_2$. Since $H$ has six 4-faces, there is at least 4-face, say $C_0$, of $H$ that contains no new edge. Thus, $C_0$ is also a 4-cycle in $G$. Except $x_4$, all vertices  in $G'$ has the same degree as in $G$. Thus, $C_0$ has at least two 3-vertices, contradicting Corollary~\ref{cor:4face-one-4ver}. In summary, $p(G') = q(G') = 0$. Hence, $G'$ has an induced forest $F'$ of order at least $an(G') - bm(G')$. Let $F = F' \cup \{w_0,w_3\}$. $F$ is an induced forest of $G$ of order at least $an(G') -bm(G') + 2$. Since $n(G') = n(G) - 3$ and $m(G') = m(G)-5$, we have:
\begin{equation*}
|F| \geq an(G) -bm(G) + 2 - 3 a + 5b
\end{equation*}
Since $2-3a + 5b$ is non-negative by Inequality~\eqref{lpc:2v-2vb}, $|F| \geq an(G) -bm(G)$, contradicting that $G$ is a counter-example.
\end{proof}
\paragraph{Proof of Theorem~\ref{thm:sub-main}} We have shown that if $a,b,c,d$ satisfy all constraints in $\lp$, a counter-example graph $G$ must be two-connected, have $\delta(G) \geq 3$, have no 4-face with at least one 3-vertex and have no 5-face with at least four 3-vertices, contradicting Theorem~\ref{thm:structure}. To finish the proof of Theorem~\ref{thm:sub-main},  we only need to show that Linear Program $\lp$ that consists of constraints from~\eqref{lpc:Q-3} to~\eqref{lpc:4f-1-3v} is equivalent to Linear Program~\eqref{eq:our-linear}. We observe that the set of constraints in Linear Program~\eqref{eq:our-linear} is a subset of the set of constraints in $\lp$ since:

\begin{equation*}
\begin{split}
(\text{\ref{seq:our-triv3}}) &= (\text{\ref{lpc:4f-2va}}),(\text{\ref{seq:our-triv4}}) = (\text{\ref{lpc:4f-2vb}}), (\text{\ref{seq:our-deg-5}}) = (\text{\ref{lpc:deg-5}}), (\text{\ref{seq:Q-3-ex}}) = (\text{\ref{lpc:Q-3}}), (\text{\ref{seq:T-3-ex}}) = (\text{\ref{lpc:T-6}}), (\text{\ref{seq:Q-3-1-ex}}) = (\text{\ref{lpc:Q-3-1a}}), (\text{\ref{seq:Q-3-1-ex1}}) = (\text{\ref{lpc:Q-3-1b}})\\
(\text{\ref{seq:T-6-1-ex}}) &= (\text{\ref{lpc:T-6-1}}), (\text{\ref{seq:Q-3-2-ex}}) = (\text{\ref{lpc:Q-3-2a}}), (\text{\ref{seq:Q-3-2-ex1}}) = (\text{\ref{lpc:Q-3-2b}}), (\text{\ref{seq:Q-3-3-ex}}) = (\text{\ref{lpc:Q-3-3}}), (\text{\ref{seq:T-3-45-ex}}) = (\text{\ref{lpc:T-6-4-5a}}), (\text{\ref{seq:T-3-45-ex1}}) = (\text{\ref{lpc:T-6-4-5b}}), (\text{\ref{seq:ex-2-2}}) = (\text{\ref{lpc:2v-2va}}) 
\end{split}
\end{equation*}

Here we note that Inequality~\eqref{seq:T-6-1-ex} is equivalent to Inequality~\eqref{lpc:T-6-1} when $t = 0$. Remaining constraints of $\lp$, we express as linear combinations of constraints in Linear Program~\ref{eq:our-linear} as follows: 

%\begin{equation*}
%\begin{split}
\begin{align*}
(5t+4)-(8t+6)a + (13t+9)b - d ~(\text{\ref{lpc:T-6-1}}) &=  (\text{\ref{seq:T-6-1-ex}}) + t(\text{\ref{seq:Q-3-1-ex}})\\
4-6a + 10b - d ~(\text{\ref{lpc:T-6-2}})&=  (\text{\ref{seq:T-6-1-ex}}) + (\text{\ref{seq:our-triv}})\\
5 - 6a + 16b - c - d  ~(\text{\ref{lpc:Q-3-4}})&=  (\text{\ref{seq:Q-3-3-ex}}) +  (\text{\ref{seq:our-triv}})\\
4 - 6a + 11b - d ~(\text{\ref{lpc:T-6-3}}) &=  (\text{\ref{seq:T-6-1-ex}}) + 2(\text{\ref{seq:our-triv}})\\
5 - 8a + 17b  - c -d~(\text{\ref{lpc:Q-3-5}}) &=  (\text{\ref{seq:Q-3-3-ex}}) + 2 (\text{\ref{seq:our-triv}})\\
 1 - 2a + 5b~(\text{\ref{lpc:2v-4v}}) &=  (\text{\ref{seq:our-deg-5}}) + (\text{\ref{seq:our-iso1}})\\
 2 - 3a + 5b ~(\text{\ref{lpc:2v-2vb}})&=   2(\text{\ref{seq:our-iso1}}) + (\text{\ref{seq:our-deg-5}})\\
 3 - 5a + 9b ~(\text{\ref{lpc:min-deg-3a}})&=   (\text{\ref{seq:ex-2-2}}) + (\text{\ref{seq:our-deg-5}})\\
 4 - 6a + 9b ~(\text{\ref{lpc:min-deg-3b}})&=   (\text{\ref{seq:T-6-1-ex}}) + (\text{\ref{seq:our-triv2}})  \\
 6 - 9a + 14b ~(\text{\ref{lpc:3v-4va}})&=   (\text{\ref{seq:Q-3-1-ex}}) + (\text{\ref{seq:our-triv}}) + (\text{\ref{seq:our-iso1}})\\
 1 + 5b - 2a ~(\text{\ref{lpc:3v-4vb}})&= (\text{\ref{seq:our-deg-5}}) + (\text{\ref{seq:our-iso1}})\\
 4 - 7a + 14b ~(\text{\ref{lpc:sep4-three-3v}})&=   (\text{\ref{seq:our-iso1}}) + 2(\text{\ref{seq:our-deg-5}}) + (\text{\ref{seq:ex-2-2}})\\
 7 - 11a + 18b ~(\text{\ref{lpc:2e-one-4f}})&= (\text{\ref{seq:our-iso1}}) + 2(\text{\ref{seq:our-deg-5}}) +(\text{\ref{seq:ex-2-2}}) \\
 3 - 6a + 14b~(\text{\ref{lpc:2v-4vX}})&=   (\text{\ref{seq:ex-2-2}}) + 2(\text{\ref{seq:our-deg-5}})\\
 7 - 12a + 23b ~ (\text{\ref{lpc:4f-2-3v}})&= 2(\text{\ref{seq:ex-2-2}}) + 3(\text{\ref{seq:our-deg-5}})+(\text{\ref{seq:our-iso1}})\displaybreak[3]\\
 3 - 6a + 15b - c ~(\text{\ref{lpc:xyz-adj}}) &= (\text{\ref{seq:our-deg-5}}) + (\text{\ref{seq:T-3-45-ex}})\\
 4 -8a + 19b ~(\text{\ref{lpc:x-y-4v}})&= 3(\text{\ref{seq:our-deg-5}}) + (\text{\ref{seq:ex-2-2}}) + (\text{\ref{seq:our-iso1}})\\
 3-7a + 19b ~(\text{\ref{lpc:4f-1-3v}})&= (\text{\ref{seq:ex-2-2}}) + 3(\text{\ref{seq:our-deg-5}})
\end{align*}
%\end{split}
%\end{equation*}

\section{Conclusion}
We have introduced a new approach that can  handle special graphs of small order separately to find an induced forest of order at least $\frac{5n}{9}$ in triangle-free planar graphs of order $n$. It would be very interesting to see whether our method can be employed to give a better bound on the order of the largest induced forest in girth-5 planar graphs~\cite{DMP14} and the order of the induced forest in subcubic (non-planar) graphs of girth at least four and five~\cite{KL16}. Another direction is to improve our analysis to obtained $\frac{4n}{7}$ bound. This would match the bound obtained by Wang, Xie and Yu~\cite{WXY16} for bipartite planar graphs and would possibly give a simpler and more general proof than the  proof by Wang, Xie and Yu. The ultimate goal is to resolve the conjecture of Akiyama and Watanabe and we would like to see if our method can be extended to resolve this conjecture as well. 

\paragraph*{Acknowledgment.} We thank Baigong Zheng for proofreading this paper. We thank conversations with Glencora Borradaile and Melissa Sherman-Bennett during the development of this work. We also would like to thank Bojan Mohar for pointing out mistakes in the statement of Theorem~\ref{thm:main} in earlier versions of this paper. This material is based upon work supported by the National Science Foundation under Grant No.\ CCF-1252833.
\nocite{KLS10, DMP14,Salavatipour06, AW87,LMZ15,WXY16,AB79, Borodin79,AMT01, CW13}
 \bibliographystyle{plain}
 \bibliography{inducedforest}
 
\end{document}